\documentclass[11pt,letterpaper]{amsart}
\usepackage{amssymb
,amsthm
,amsmath
,amscd
,mathtools
,mathdots
}
\usepackage[all]{xy}

\newcommand{\C}{\mathbb{C}}
\newcommand{\Q}{\mathbb{Q}}
\newcommand{\Z}{\mathbb{Z}}
\newcommand{\W}{\mathbb{W}}

\newcommand{\LL}{\mathcal{L}}
\newcommand{\TT}{\mathcal{T}}
\newcommand{\Sc}{\mathcal{S}}
\newcommand{\WW}{\mathcal{W}}

\newcommand{\GL}{\mathrm{GL}}
\newcommand{\SL}{\mathrm{SL}}
\newcommand{\Sp}{\mathrm{Sp}}
\newcommand{\U}{\mathrm{U}}
\newcommand{\M}{\mathrm{M}}

\newcommand{\Irr}{\mathrm{Irr}}
\newcommand{\temp}{\mathrm{temp}}
\newcommand{\diag}{\mathrm{diag}}
\newcommand{\Gal}{\mathrm{Gal}}
\newcommand{\vol}{\mathrm{vol}}
\newcommand{\id}{\mathrm{id}}
\newcommand{\tr}{\mathrm{tr}}
\newcommand{\Hom}{\mathrm{Hom}}
\newcommand{\Supp}{\mathrm{Supp}}
\newcommand{\Ind}{\mathrm{Ind}}
\newcommand{\ind}{\mathrm{ind}}
\newcommand{\ord}{\mathrm{ord}}
\newcommand{\re}{\mathrm{Re}}
\newcommand{\Herm}{\mathrm{Herm}}

\newcommand{\iif}{&\quad&\text{if }}
\newcommand{\resp}{resp.~}
\renewcommand{\1}{\mathbf{1}}
\newcommand{\ep}{\varepsilon}
\newcommand{\bs}{\backslash}

\newcommand{\oo}{\mathfrak{o}}
\newcommand{\pp}{\mathfrak{p}}
\newcommand{\bv}{\mathbf{v}}
\newcommand{\bm}{\mathbf{m}}
\newcommand{\bn}{\mathbf{n}}

\newcommand{\half}[1]{\frac{#1}{2}}
\newcommand{\pair}[1]{\left\langle #1 \right\rangle}
\newcommand{\tl}[1]{\widetilde{#1}}
\newcommand{\ub}[1]{\underline{#1}}

\newtheorem{thm}{Theorem}[section]
\newtheorem{lem}[thm]{Lemma}
\newtheorem{prop}[thm]{Proposition}
\newtheorem{rem}[thm]{Remark}

\title[Local newforms for $\U_{2n}$]
{Local newforms for generic representations of unramified even unitary groups I: 
\\ Even conductor case}

\author{
Hiraku Atobe
}
\date{}

\address{
Department of Mathematics, Hokkaido University,
Kita 10, Nishi 8, Kita-Ku, Sapporo, Hokkaido, 060-0810, Japan 
}

\email{
atobe@math.sci.hokudai.ac.jp
}

\allowdisplaybreaks
\begin{document}
\maketitle

\begin{abstract}
In this paper, we define compact open subgroups of quasi-split even unitary groups for each even non-negative integers, 
and establish the theory of local newforms for irreducible tempered generic representations
with a certain condition on the central characters. 
To do this, we use the local Gan--Gross--Prasad conjecture, 
the local Rankin--Selberg integrals, and the local theta correspondence.
\end{abstract}

\tableofcontents

\section{Introduction}
In 1970's, Atkin--Lehner \cite{AL} and Li \cite{Li} introduced 
the notion of \emph{newforms} for elliptic modular forms, 
and showed the multiplicity one theorem.
Together with their results, Casselman's theory of \emph{local newforms} \cite{C}
is a bridge between modular forms 
and automorphic representations of $\GL_2/\Q$.
Since then, the theory of local newforms was developed for several groups. 
For example, for low rank cases, 
Roberts--Schmidt \cite{RS1} and Lansky--Raghuram \cite{LR}
established this theory for $\mathrm{GSp}_4$ and $\U(1,1)$, respectively. 
Casselman's result was extended to $\GL_n$ by Jacquet--Piatetski-Shapiro--Shalika \cite{JPSS} (see also \cite{J})
and by Atobe--Kondo--Yasuda \cite{AKY}. 
For other general rank cases, 
\begin{itemize}
\item
Tsai \cite{Ts} studied the local newforms of generic supercuspidal representations of $\mathrm{SO}_{2n+1}$; and 
\item
the author together with Oi and Yasuda \cite{AOY} treated the case for $\U_{2n+1}$.
\end{itemize}
In this paper, for a bridge to hermitian modular forms, 
we try to establish the theory of local newforms for $\U(n,n)$.
\vskip 10pt

Let us describe our results.
Let $E/F$ be an unramified quadratic extension 
of non-archimedean local fields of characteristic $0$ and of residue characteristic $p > 2$.
Fix a non-trivial additive character $\psi$ of $F$ 
such that $\psi|_{\oo_F} = \1$ but $\psi|_{\pp_F^{-1}} \not= \1$, 
and set $\psi_E(x) = \psi(\half{x+\overline{x}})$ for $x \in E$.
Consider a quasi-split unitary group of $2n$ variables given by 
\[
\U_{2n} = \left\{
g \in \GL_{2n}(E) \;\middle|\;
{}^t\overline{g} 
\begin{pmatrix}
0&w_n \\ 
-w_n &0
\end{pmatrix}
g = 
\begin{pmatrix}
0&w_n \\ 
-w_n &0
\end{pmatrix}
\right\}
\]
with 
\[
w_n = \begin{pmatrix}
&&1 \\
&\iddots&\\
1&&
\end{pmatrix} 
\in \GL_n(E).
\]
The center of $\U_{2n}$ is identified with $E^1 = \{x \in E^\times \;|\; N_{E/F}(x) = 1\}$.
Define a compact subgroup $K_{2m}^W$ of $\U_{2n}$ by 
$K_{0}^W = \U_{2n} \cap \GL_{2n}(\oo_E)$, 
and by 
\[
K_{2m}^W = 
\bordermatrix{
&1&2n-2&1 \cr
1&1+\pp_E^m&\oo_E&\oo_E\cr
2n-2&\pp_E^m&\oo_E&\oo_E\cr
1&\pp_E^{2m}&\pp_E^m&1+\pp_E^m
} \cap \U_{2n}
\]
for $2m > 0$. 
For an irreducible smooth representation $\pi$ of $\U_{2n}$,
we denote by $\pi_\psi$ the maximal quotient of $\pi$
on which the subgroup 
\[
Z = \left\{
\begin{pmatrix}
1&0& z \\
0&\1_{2n-2}&0\\
0&0&1
\end{pmatrix}
\in \U_{2n}
\;\middle|\; z \in F
\right\} 
\cong F
\]
acts by $\psi$.
This is a local analogue of the Fourier--Jacobi expansions of hermitian modular forms, 
and is called the \emph{Fourier--Jacobi module} of $\pi$. 
We write $\pi_\psi^{K_{2m}^W}$ for the image of the subspace $\pi^{K_{2m}^W}$ consisting of $K_{2m}^W$-fixed vectors
via the canonical surjection $\pi \twoheadrightarrow \pi_\psi$. 
\vskip 10pt

The main theorem is stated as follows.
For other notations, see Section \ref{sec.statement} below.

\begin{thm}[Theorem \ref{even}]\label{main}
Let $\pi$ be an irreducible tempered representation of $\U_{2n}$
with the $L$-parameter $\phi_\pi$ and the central character $\omega_\pi$. 
We denote by $c(\phi_\pi)$ the conductor of $\phi_\pi$. 

\begin{enumerate}
\item
If $\pi$ is not $\psi_E$-generic, 
then $\pi_{\psi}^{K_{2m}^W} = 0$ for any $2m \geq 0$. 
Conversely, if $\pi$ is $\psi_E$-generic, then there exists $2m \geq 0$ such that $\pi_{\psi}^{K_{2m}^W} \not= 0$. 

\item
Suppose that $\pi$ is $\psi_E$-generic. 
If $2m < c(\phi_\pi)$, then $\pi_{\psi}^{K_{2m}^W} = 0$. 
If $2m = c(\phi_\pi)$ or $2m = c(\phi_\pi)+1$, then 
\[
\dim_\C(\pi_{\psi}^{K_{2m}^W}) \leq 1.
\]

\item
Set $2m = c(\phi_\pi)$ or $2m = c(\phi_\pi)+1$.
Suppose that $\pi$ is $\psi_E$-generic and that $\omega_\pi$ is trivial on $E^1 \cap (1+\pp_E^m)$.
Then $\pi_\psi^{K_{2m}^W} \not= 0$. 

\end{enumerate}
\end{thm}

If $2m = c(\phi_\pi)$ and if $\omega_\pi$ is trivial on $E^1 \cap (1+\pp_E^m)$, 
we shall call an element in $\pi^{K_{2m}^W}$ whose image in $\pi_\psi$ is nonzero
a \emph{local newform} of $\pi$.

\begin{rem}
\begin{enumerate}
\item
If $\pi^{K_{2m}^W} \not= 0$, then $\omega_\pi$ is trivial on $E^1 \cap (1+\pp_E^m)$
since $E^1 \cap (1+\pp_E^m) \subset K_{2m}^W$.

\item
Even if $2m = c(\phi_\pi)$ or $c(\phi_\pi)+1$, 
the dimension of $\pi^{K_{2m}^W}$ can be greater than $1$.
A counterexample already appears in the case where $n=1$, 
which was treated by Lansky and Raghuram. 
See \cite[Theorem 4.2.1]{LR}. 

\item
As well as in \cite{AOY}, there might exist $K_m^W$ for odd integer $m >0$. 
Unfortunately, the author could not find it.
\end{enumerate}
\end{rem}

We expect that Theorem \ref{main} has several applications such as 
a higher level generalization of a result of Chenevier--Renard \cite{CR}. 
We will try it as a next project. 
\vskip 10pt

A usual method to establish the theory of local newforms is 
to apply the \emph{Rankin--Selberg integrals}, 
which are based on the multiplicity one theorem for 
several \emph{Gan--Gross--Prasad (GGP) pairs}.
For example, Tsai \cite{Ts} and Cheng \cite{Cheng} 
used the pairs $(\mathrm{SO}_{2n+1}(F), \mathrm{SO}_{2n}(F))$ and $(\U_{2n+1},\U_{2n})$
to obtain knowledge about newforms.
In this paper, we will also use this method as well. 
However, in our case, 
one needs the GGP pair $(\U_{2n}, \U_{2n-2})$, which is not a ``basic'' case.
More precisely, we have to consider the restrictions of irreducible representations of $\U_{2n}$ to the \emph{Jacobi group}.
Since the Jacobi group is not reductive, several arguments in \cite{Ts} would not work.
\vskip 10pt

For example, 
to prove an analogue of Theorem \ref{main} (1) in \cite{Ts}, Tsai used a lemma of Moy--Prasad (\cite[Lemma 3.4.1]{Ts}).
We do not know whether this lemma can be extended to our case. 
Instead of this lemma, 
we use the local period integrals for the refined GGP conjecture. 
Using the absolutely convergence of these integrals, 
the argument of Gan--Savin \cite[Lemma 12.5]{GS} can show Theorem \ref{main} (1).
See Section \ref{sec.proof1} below.
This is the same idea as in the previous paper \cite[Theorem 4.5]{AOY}.
\vskip 10pt

The proof of Theorem \ref{main} (2) is the same as usual. 
Namely, it is an application of the Rankin--Selberg integrals for $\U_{2n} \times \GL_{n-1}(E)$.
This theory in this case was established by Ben-Artzi--Soudry \cite{BS} and Morimoto \cite{Mori}, 
and is recalled in Theorem \ref{RS}. 
Especially, the multiplicativity of the gamma factors is included in \cite[Theorem 3.1]{Mori}.
Using the Rankin--Selberg integrals, we will define certain formal power series.
Lemma \ref{degree} is a key computation to give lower bounds of the degrees. 
Using the functional equations of the Rankin--Selberg integrals, 
we would obtain an upper bound of the dimension of $\pi^{K_{2m}^W}_\psi$. 
However, since the Rankin--Selberg integrals for $\U_{2n} \times \GL_{n-1}(E)$
factors through $\pi \twoheadrightarrow \pi_\psi$, 
we cannot estimate the dimension of $\pi^{K_{2m}^W}$ itself. 
\vskip 10pt

For the proof of Theorem \ref{main} (3), 
the fact that we have to deal with the Jacobi group complicates the situation.
Indeed, the arguments in \cite[Chapter 8]{Ts} and in the previous paper \cite[Theorem 4.3]{AOY} might not work.
In this paper, we give a new, or rather old, idea.
\vskip 10pt

Recall that the theory of newforms was initiated by Atkin--Lehner \cite{AL} and Li \cite{Li} 
for elliptic modular forms of integral weights.
Kohnen \cite{Kohnen} established a similar theory to the half-integral weights case.
Moreover, he proved that 
the newforms of integral weights and the ones of half-integral weights 
are related to each other by the \emph{Shimura correspondence}. 
Since the \emph{theta correspondence} is a generalization of the Shimura correspondence, 
the local newforms will be compatible with the local theta correspondence in the future.
Instead, the local theta correspondence would be useful to show the existence of the local newforms.
This is our idea.
\vskip 10pt

In fact, if we let $\sigma = \theta_\psi(\pi)$ be the theta lift of $\pi$ to $\U_{2n+1}$, 
then $\sigma$ is nonzero irreducible tempered and generic, and 
its conductor and central character are the same as the ones of $\pi$. 
By the definition of the theta lifting, 
we have a surjective $\U_{2n+1} \times \U_{2n}$-equivariant map
\[
\omega_\psi \rightarrow \sigma \boxtimes \pi, 
\]
where $\omega_\psi$ is the Weil representation of $\U_{2n+1} \times \U_{2n}$. 
Let $K_{2m}^V$ be a conjugate of the compact subgroup of $\U_{2n+1}$ defined in \cite{AOY}. 
Set $J_{2m}^V$ to be the subgroup of $\U_{2n+1}$ generated by $K_{2m}^V$ and the central subgroup $E^1 \cap (1+\pp_E^m)$. 
Then by using a lattice model and Waldspurger's result (Proposition \ref{WK}), 
one can show that $\omega_\psi^{J_{2m}^V}$ is generated by $\omega_\psi^{J_{2m}^V \times K_{2m}^W}$ as a representation of $\U_{2n}$. 
Hence if $2m \geq c(\phi_\pi)$ and $\omega_\pi|_{E^1 \cap (1+\pp_E^m)} = \1$, 
then $\pi^{K_{2m}^W} \not= 0$ since $\sigma^{J_{2m}^V} \not= 0$.
See Proposition \ref{K-fixed} for the details.
\vskip 10pt

On the other hand, it is much harder to show $\pi^{K_{2m}^W}_\psi \not= 0$ 
when $2m = c(\phi_\pi)$ or $2m = c(\phi_\pi)+1$.
Let $l_\sigma \colon \sigma \rightarrow \C$ be a nonzero Whittaker functional.
Then the composition 
\[
\omega_\psi \rightarrow \sigma \boxtimes \pi \xrightarrow{l_\sigma \otimes \id} \pi
\]
factors through a twisted Jacquet module of $\omega_\psi$ 
along a maximal unipotent subgroup of $\U_{2n+1}$.
By the same argument as Mao--Rallis \cite[Proposition 2.3]{MR}, 
this twisted Jacquet module is isomorphic to the compact induction $\ind_{N_{2n}'}^{\U_{2n}}(\mu)$, 
where $N_{2n}'$ is a maximal unipotent subgroup of $\U_{2n}$ and $\mu$ is a generic character of $N_{2n}'$. 
By Cheng's result \cite[Theorem 1.4, Lemma 7.5]{Cheng}, 
$l_\sigma$ is nonzero on the one dimensional subspace $\sigma^{J_{2m}^V}$ 
if $l_\sigma$ is suitably chosen.
Hence there is $\phi \in \omega_\psi^{J_{2m}^V \times K_{2m}^W}$ 
such that it is nonzero under the all maps in the following diagram:
\[
\xymatrix{
\omega_\psi \ar@{->}[r] \ar@{->}[d] & \sigma \boxtimes \pi \ar@{->}[rr]^-{l_\sigma \otimes \id} && \pi \\
\ind_{N_{2n}'}^{\U_{2n}}(\mu) \ar@{-->}[rrru] &&&
}
\]
Lemma \ref{supp} asserts that the support of the image of $\phi$ in $\ind_{N_{2n}'}^{\U_{2n}}(\mu)$ is small enough.
It implies that $\pi^{K_{2m}^W}_\psi \not= 0$ immediately. 
See Section \ref{sec.proof3} for the details.
Finally, to prove Lemma \ref{supp}, we need to change models of the Weil representation, 
and review the argument of Mao--Rallis \cite[Proposition 2.3]{MR}. 
\vskip 10pt

This paper is organized as follows. 
In Section \ref{sec.statement}, 
we introduce several notations and state our main theorem.
Using the local Fourier--Jacobi periods, 
we show Theorem \ref{main} (1) in Section \ref{sec.FJ}. 
Theorem \ref{main} (2) is obtained as an application of the Rankin--Selberg integrals in 
Section \ref{sec.unique}. 
Finally, we study theta liftings to prove Theorem \ref{main} (3) in Section \ref{sec.existence}.

\subsection*{Acknowledgement}
We would like to thank Kazuki Morimoto and Yao Cheng for sending their preprints \cite{Mori} and \cite{Cheng}, respectively. 
We appreciate Ren-He Su for a private discussion 
from which we obtain 
the idea to use the theta correspondence for newforms. 
The author was supported by JSPS KAKENHI Grant Number 19K14494.

\subsection*{Notation}
Let $E/F$ be an unramified quadratic extension 
of non-archimedean local fields of characteristic $0$ and of residue characteristic $p > 2$.
The non-trivial element in $\Gal(E/F)$ is denoted by $x \mapsto \overline{x}$.
Set $\oo_E$ (\resp $\oo_F$) to be the ring of integers of $E$ (\resp $F$), 
and $\pp_E$ (\resp $\pp_F$) to be its maximal ideal. 
Let $E^1 = \{x \in E^\times \;|\; x\overline{x} = 1\}$ denote 
the kernel of the norm map $N_{E/F} \colon E^\times \rightarrow F^\times$. 
Fix a uniformizer $\varpi$ of $F$, which is also a uniformizer of $E$. 
When $x \in E^\times$ can be written as $x = u \varpi^l$ for some $u \in \oo_E^\times$, 
we write $\ord(x) = l$.
Set $q = |\oo_F/\pp_F|$ so that $q^2 = |\oo_E/\pp_E|$.
Let $|\cdot|_E$ be the normalized absolute value of $E$ so that $|x|_E = q^{-2 \ord(x)}$ for $x \in E^\times$.
\par

We fix $\delta \in \oo_E^\times$ such that $\overline{\delta} = -\delta$, 
and a non-trivial additive character $\psi \colon F \rightarrow \C^\times$
such that $\psi|_{\oo_F} = \1$ but $\psi|_{\pp_F^{-1}} \not= \1$. 
Set $\psi_E(x) = \psi(\half{1}\tr_{E/F}(x)) = \psi(\half{x+\overline{x}})$ and $\psi_E^\delta(x) = \psi_E(x/\delta)$. 
Then $\psi_E$ and $\psi_E^\delta$ are non-trivial additive characters of $E$ 
such that $\psi_E|_F = \psi$ and $\psi_E^\delta|_F = \1$.
The unique non-trivial quadratic unramified character of $E^\times$ is denoted by $\chi$.
Namely, $\chi|_{\oo_E^\times} = \1$ and $\chi(\varpi) = -1$.
In particular, if we write $\chi = |\cdot|_E^{s_0}$, we have $q^{-2s_0} = -1$.
\par

A representation $\pi$ of a $p$-adic group $G$ means 
a smooth representation over a complex vector space. 
When $K$ is a compact open subgroup of $G$, 
we write $\pi^K$ for the subspace of $\pi$ consisting of $K$-fixed vectors.
Let $\Irr(G)$ be the set of equivalence classes of irreducible representations of $G$, 
and $\Irr_\temp(G)$ be its subset consisting of tempered representations.

\section{Statement of the main theorem}\label{sec.statement}
In this section, we define families of compact open subgroups of unitary groups, 
and state our main theorem.

\subsection{Unitary groups}\label{sec.gp}
Let $V = V_{2n+1}$ (\resp $W = W_{2n}$) be 
a hermitian (\resp skew-hermitian) space over $E$ of dimension $2n+1$ (\resp $2n$)
equipped with a non-degenerate hermitian form $\pair{\cdot, \cdot}_V$
(\resp skew-hermitian form $\pair{\cdot,\cdot}_W$).
Assume that there are bases $\{e_n, \dots, e_1, e_0, e_{-1}, \dots, e_{-n}\}$ of $V$
and $\{f_n,\dots,f_1, f_{-1}, \dots, f_{-n}\}$ of $W$, respectively, 
such that
\[
\pair{e_i, e_j}_V = \pair{f_i, f_j}_W = 0
\]
unless $j = -i$, 
and 
\[
\pair{e_0,e_0}_V = \pair{e_i,e_{-i}}_V = \pair{f_i,f_{-i}}_W = 1
\]
for $1 \leq i \leq n$.  
\par

Using these bases, 
we often identify the associated unitary groups $\U(V)$ and $\U(W)$ with 
\begin{align*}
\U_{2n+1} &= \left\{
h \in \GL_{2n+1}(E) \;\middle|\;
{}^t\overline{h} w_{2n+1} h = w_{2n+1}
\right\}, \\
\U_{2n} &= \left\{
g \in \GL_{2n}(E) \;\middle|\;
{}^t\overline{g} J_{2n} g = J_{2n}
\right\}, 
\end{align*}
respectively, 
where we set
\[
w_n = \begin{pmatrix}
&&1\\
&\iddots&\\
1&&
\end{pmatrix}
\in \GL_n(E),
\quad
J_{2n} = 
\begin{pmatrix}
0&w_n\\
-w_n&0
\end{pmatrix} 
\in \GL_{2n}(E).
\]

\subsection{Representations of unitary groups}\label{sec.rep}
Let $N_{2n+1}$ (\resp $N_{2n}$) be the upper triangular unipotent subgroup of $\U_{2n+1}$ (\resp $\U_{2n}$).
We define generic characters of $N_{2n+1}$ and $N_{2n}$ by the same formula
\[
u \mapsto \psi_E\left(\sum_{k=1}^{n}u_{k,k+1}\right). 
\]
By abuse of notation, we denote these characters by $\psi_E$.
We say that an irreducible representation $\sigma$ of $\U_{2n+1}$ (\resp $\pi$ of $\U_{2n}$)
is \emph{generic} (\resp \emph{$\psi_E$-generic}) 
if $\Hom_{N_{2n+1}}(\sigma, \psi_E) \not= 0$ 
(\resp $\Hom_{N_{2n}}(\pi, \psi_E) \not= 0$).
\par

For an irreducible representation $\pi$ of $\U_{2n}$, 
we denote by $\pi^\vee$ the contragredient representation of $\pi$. 
By a result in \cite[Chapter 4. II. 1]{MVW}, we know $\pi^\vee \cong \pi^\theta$, 
where $\pi^\theta(g) = \pi(\theta(g))$ with 
\[
\theta \colon \U_{2n} \rightarrow \U_{2n},\; 
g \mapsto 
\begin{pmatrix}
\1_n & 0 \\ 0 & -\1_n
\end{pmatrix}
\overline{g}
\begin{pmatrix}
\1_n & 0 \\ 0 & -\1_n
\end{pmatrix}^{-1}.
\]
In particular, $\pi$ is $\psi_E$-generic if and only if $\pi^\vee$ is $\psi_E^{-1}$-generic.
\par

By the local Langlands correspondence established by Mok \cite{Mok}, 
to an irreducible representation $\sigma$ of $\U_{2n+1}$ (\resp $\pi$ of $\U_{2n}$), 
one can attach a conjugate self-dual representation $\phi_\sigma$ (\resp $\phi_\pi$)
of $W_E \times \SL_2(\C)$ of dimension $2n+1$ (\resp $2n$), 
where $W_E$ is the Weil group of $E$.
We call $\phi_\sigma$ (\resp $\phi_\pi$) 
the \emph{$L$-parameter} for $\sigma$ (\resp $\pi$).
Then we define the \emph{conductor} $c(\phi_\sigma)$ of $\phi_\sigma$ 
by the non-negative integer satisfying 
\[
\ep(s,\phi_\sigma,\psi_E)
= 
\ep(0,\phi_\sigma,\psi_E)q^{-2c(\phi_\sigma)s}.
\]
Similarly, the \emph{conductor} $c(\phi_\pi)$ of $\phi_\pi$ is defined.
\par

The center of $\U_{2n+1}$ (\resp $\U_{2n}$) is $\U_1$ which is identified with $E^1$.
For an irreducible representation $\sigma$ (\resp $\pi$) of $\U_{2n+1}$ (\resp $\U_{2n}$), 
we denote its central character by $\omega_\sigma$ (\resp $\omega_\pi$). 
If $\sigma$ (\resp $\pi$) corresponds to $\phi_\sigma$ (\resp $\phi_\pi$), 
then the $L$-parameter of $\omega_\sigma$ (\resp $\omega_\pi$) is given by 
$\det(\phi_\sigma)$ (\resp $\det(\phi_\pi)$).

\subsection{Jacobi group}\label{sec.jacobi}
Set 
\[
\bv(x,y;z) = 
\begin{pmatrix}
1 & x & y & z+\half{1}(x w_{n-1}{}^t\overline{y} - y w_{n-1} {}^t\overline{x})\\
0 & \1_{n-1} & 0 & w_{n-1}{}^t\overline{y} \\
0 & 0& \1_{n-1} & -w_{n-1}{}^t\overline{x} \\
0&0&0& 1
\end{pmatrix}
\in \U_{2n}
\]
for $x,y \in E^{n-1}$ and $z \in F$. 
Here, $E^{n-1}$ is the space of row vectors.
Let $H_{n-1} = \{\bv(x,y; z) \;|\; x,y \in E^{n-1}, z \in F\} \cong E^{2n-2} \oplus F$ 
be a Heisenberg group in $4n-3$ variables over $F$
with the multiplication law
\[
\bv(x,y;z) \bv(x',y';z') = 
\bv\left(x+x', y+y'; z+z'+\half{1}\tr_{E/F}(x w_{n-1}{}^t\overline{y} - y w_{n-1} {}^t\overline{x}) \right).
\]
We write
\begin{align*}
X_{n-1} &= \{\bv(x,0;0) \;|\; x \in E^{n-1}\}, \\
Y_{n-1} &= \{\bv(0,y;0) \;|\; y \in E^{n-1}\}, \\
Z &= \{\bv(0,0;z) \;|\; z \in F\}.
\end{align*}
By abuse of notation, we denote the character
$Z \ni \bv(0,0;z) \mapsto \psi(z)$ by $\psi$. 
\par

We identify $\U_{2n-2}$ as a subgroup of $\U_{2n}$ by the inclusion
\[
\U_{2n-2} \ni g' \mapsto \begin{pmatrix}
1 && \\ 
&g'& \\
&&1
\end{pmatrix} \in \U_{2n}.
\]
Then $\U_{2n-2}$ normalizes $H_{n-1}$.
We call $J_{n-1} = H_{n-1} \rtimes \U_{2n-2}$ the \emph{Jacobi group}.
Note that $Z$ is the center of $J_{n-1}$. 
\par

For an irreducible representation $\pi$ of $\U_{2n}$, 
we denote by $\pi_{\psi}$ the maximal quotient of $\pi$
on which $Z$ acts by $\psi$. 
We call $\pi_\psi$ the \emph{Fourier--Jacobi module} of $\pi$.
For a compact open subgroup $K$ of $\U_{2n}$, 
we denote by $\pi^K_{\psi}$ the image of $\pi^K$ 
via the canonical surjection $\pi \twoheadrightarrow \pi_{\psi}$. 
Note that $\pi_{\psi}$ is a smooth representation of $J_{n-1}$ 
so that $K$ does not act on $\pi_{\psi}$ itself. 
\par

For $t \in E^\times$, if we put $\psi'(x) = \psi(N_{E/F}(t)x)$ and 
\[
K' = 
\begin{pmatrix}
t && \\ 
&\1_{2n-2}& \\ 
&& \overline{t}^{-1}
\end{pmatrix}^{-1}
K
\begin{pmatrix}
t && \\ 
&\1_{2n-2}& \\ 
&& \overline{t}^{-1}
\end{pmatrix}, 
\]
then $\pi(\diag(t,\1_{2n-2}, \overline{t}^{-1}))$ induces isomorphisms 
\[
\pi^{K'} \xrightarrow{\sim} \pi^K,\quad
\pi_{\psi'} \xrightarrow{\sim} \pi_\psi.
\]
Hence we have $\pi_{\psi'}^{K'} \cong \pi_{\psi}^K$.

\subsection{Compact subgroups}\label{sec.cpt}
For each non-negative even integer $2m \geq 0$, 
we define compact subgroups $K_{2m}^V \subset \U(V) \cong \U_{2n+1}$ 
and $K_{2m}^W \subset \U(W) \cong \U_{2n}$ as follows. 
When $2m = 0$, we set $K_0^V = \U_{2n+1} \cap \GL_{2n+1}(\oo_E)$ and 
$K_0^W = \U_{2n} \cap \GL_{2n}(\oo_E)$. 
If $2m > 0$, we set 
\begin{align*}
K_{2m}^V &= 
\bordermatrix{
&n&1&n \cr
n&\oo_E&\pp_E^m&\oo_E \cr
1&\pp_E^m&1+\pp_E^{2m}&\pp_E^m \cr
n&\oo_E&\pp_E^m&\oo_E
} 
\cap \U_{2n+1}, \\
K_{2m}^W &= 
\bordermatrix{
&1&2n-2&1 \cr
1&1+\pp_E^m&\oo_E&\oo_E\cr
2n-2&\pp_E^m&\oo_E&\oo_E\cr
1&\pp_E^{2m}&\pp_E^m&1+\pp_E^m
} \cap \U_{2n}.
\end{align*}
Note that 
\begin{align*}
&\begin{pmatrix}
\varpi^{-m} \cdot \1_n &&\\
&1&\\
&&\varpi^{m} \cdot \1_n
\end{pmatrix}
K_{2m}^V
\begin{pmatrix}
\varpi^{-m} \cdot \1_n &&\\
&1&\\
&&\varpi^{m} \cdot \1_n
\end{pmatrix}^{-1} 
\\&= 
\bordermatrix{
&n&1&n \cr
n&\oo_E&\oo_E&\pp_E^{-2m}\cr
1&\pp_E^{2m}&1+\pp_E^{2m}&\oo_E\cr
n&\pp_E^{2m}&\pp_E^{2m}&\oo_E
} \cap \U_{2n+1}, 
\end{align*}
which is denoted by $\mathbb{K}_{2m,\U(V)}$ in \cite{AOY}, and by $K_{n,2m}$ in \cite{Cheng}.
If we set ${}^tK_{2m}^W = \{{}^tk \;|\; k \in K_{2m}^W\}$ to be the transpose of $K_{2m}^W$, 
then
\begin{align*}
K_{2m}^W = 
\begin{pmatrix}
\varpi^{-m} && \\ 
&\1_{2n-2}& \\ 
&& \varpi^m
\end{pmatrix}
{}^tK_{2m}^W
\begin{pmatrix}
\varpi^{-m} && \\ 
&\1_{2n-2}& \\ 
&& \varpi^m
\end{pmatrix}^{-1}.
\end{align*}
\par

The theory of local newforms for $\U_{2n+1}$ is established 
by the author together with Oi and Yasuda \cite[Theorem 1.1]{AOY}
and by Cheng \cite[Theorem 1.2]{Cheng} as follows.

\begin{thm}\label{odd}
Let $\sigma$ be an irreducible tempered representation of $\U_{2n+1}$
with the $L$-parameter $\phi_\sigma$. 
\begin{enumerate}
\item
If $\sigma$ is not generic, then $\sigma^{K_{2m}^V} = 0$ for any $2m \geq 0$. 
\item
If $\sigma$ is generic, then
\[
\dim_\C(\sigma^{K_{2m}^V}) = 
\left\{
\begin{aligned}
&0 \iif 2m < c(\phi_\sigma), \\
&1 \iif \text{$2m = c(\phi_\sigma)$ or $c(\phi_\sigma)+1$}.
\end{aligned}
\right. 
\]
Moreover, if $2m > c(\phi_\sigma)$, then $\sigma^{K_{2m}^V} \not= 0$.
\end{enumerate}
\end{thm}

In this paper, we will prove an analogue of this theorem for $\U_{2n}$ as follows.
\begin{thm}\label{even}
Let $\pi$ be an irreducible tempered representation of $\U_{2n}$
with the $L$-parameter $\phi_\pi$ and the central character $\omega_\pi$. 

\begin{enumerate}
\item
If $\pi$ is not $\psi_E$-generic, 
then $\pi_{\psi}^{K_{2m}^W} = 0$ for any $2m \geq 0$. 
Conversely, if $\pi$ is $\psi_E$-generic, then there exists $2m \geq 0$ such that $\pi_{\psi}^{K_{2m}^W} \not= 0$. 

\item
Suppose that $\pi$ is $\psi_E$-generic. 
If $2m < c(\phi_\pi)$, then $\pi_{\psi}^{K_{2m}^W} = 0$. 
If $2m = c(\phi_\pi)$ or $2m = c(\phi_\pi)+1$, then 
\[
\dim_\C(\pi_{\psi}^{K_{2m}^W}) \leq 1.
\]

\item
Set $2m = c(\phi_\pi)$ or $2m = c(\phi_\pi)+1$.
Suppose that $\pi$ is $\psi_E$-generic and that $\omega_\pi$ is trivial on $E^1 \cap (1+\pp_E^m)$.
Then $\pi_\psi^{K_{2m}^W} \not= 0$. 
\end{enumerate}
\end{thm}

When $2m = c(\phi_\pi)$, 
we shall call an element in $\pi^{K_{2m}^W}$ whose image in $\pi_\psi$ is nonzero
a \emph{local newform} of $\pi$.

\section{Local Fourier--Jacobi periods}\label{sec.FJ}
In this section, we will prove Theorem \ref{even} (1). 
To do this, we use the local Gan--Gross--Prasad conjecture for $(\U_{2n}, \U_{2n-2})$.

\subsection{Weil representation}
Let $W_0$ be the subspace of $W$ 
generated by $\{f_{n-1},\dots,f_1, f_{-1}, \dots, f_{-n+1}\}$.
We write $G_n = \U(W)$ and $G_{n-1} = \U(W_0)$ in this section.
Hence the Jacobi group $J_{n-1}$ is written as $J_{n-1} = H_{n-1} \rtimes G_{n-1}$. 
\par

Recall that we have a compact subgroup $K_{2m}^W$ of $G_n = \U(W)$. 
Note that the intersections
\[
K^J = K_{2m}^{W} \cap J_{n-1},
\quad
K^H = K_{2m}^{W} \cap H_{n-1},
\quad
K^{W_0} = K_{2m}^{W} \cap \U(W_0)
\]
are independent of $2m$. 
Moreover, $K^{W_0}$ is a hyperspecial maximal compact subgroup of $G_{n-1} = \U(W_0)$.
\par

We consider the Weil representation $\omega_\psi$ of $J_{n-1}$ associated to $\psi$ and $\chi$.
It is realized on the Schwartz space $\Sc(E^{n-1})$ as follows. 
For $\phi \in \Sc(E^{n-1})$ and $\xi \in E^{n-1}$, 
\begin{align*}
&\omega_\psi(\bv(x,0;0))\phi(\xi) = \phi(\xi+x), \quad x \in E^{n-1}, \\
&\omega_\psi(\bv(0,y;0))\phi(\xi) = \psi_E(2\xi w_{n-1} {}^t\overline{y})\phi(\xi), \quad y \in E^{n-1}, \\
&\omega_\psi(\bv(0,0;z))\phi(\xi) = \psi(z) \phi(\xi), \quad z \in F, \\
&\omega_\psi(\bm(a))\phi(\xi) = \chi(\det(a)) |\det(a)|^{\half{1}} \phi(\xi a), \quad a \in \GL_{n-1}(E), \\
&\omega_\psi(\bn(b))\phi(\xi) = \psi_E\left(\overline{\xi} \overline{b} w_{n-1} {}^t \xi \right)\phi(\xi), 
\quad b \in \M_{n-1}(E), {}^t(w_{n-1}\overline{b}) = w_{n-1}b, \\
&\omega_\psi (J_{2n-2})\phi(\xi) 
= 
\int_{E^{n-1}} \phi(x) \psi_E(2\overline{x} \cdot {}^t\xi) dx,
\end{align*}
where we set 
\[
\bm(a) = \begin{pmatrix}
a & 0 \\ 
0 & w_{n-1}{}^t\overline{a}^{-1}w_{n-1}^{-1}
\end{pmatrix}, 
\quad
\bn(b) = \begin{pmatrix}
1 & b \\ 0 & 1
\end{pmatrix} \in G_{n-1}, 
\]
and the measure $dx$ on $E^{n-1}$ is the self-dual Haar measure with respect to $\psi_E$. 
The Weil representation $\omega_\psi$ is unitary with respect to 
the pairing 
\[
(\phi_1,\phi_2) = \int_{E^{n-1}}\phi_1(\xi) \overline{\phi_2(\xi)} d\xi.
\]
\par

Set $\phi_0 \in \Sc(E^{n-1})$ to be the characteristic function on $\oo_E^{n-1}$. 
Note that $\phi$ is fixed by $\omega_{\psi}(K^J)$. 
Moreover, the subspace $\omega_{\psi}^{K^H}$ is one dimensional spanned by $\phi_0$.

\subsection{Proof of Theorem \ref{even} (1)}\label{sec.proof1}
Let $\pi \in \Irr_\temp(G_n)$ and $\pi' \in \Irr_\temp(G_{n-1})$. 
Fix a nonzero $G_n$-invariant (\resp $G_{n-1}$-invariant) bilinear pairing 
$(\cdot, \cdot)_{\pi} \colon \pi \times \pi^\vee \rightarrow \C$
(\resp $(\cdot, \cdot)_{\pi'} \colon \pi' \times \pi'^\vee \rightarrow \C$).
For $\varphi \in \pi, \varphi^\vee \in \pi^\vee, \varphi' \in \pi', \varphi'^\vee \in \pi'^\vee$ 
and $\phi, \phi^\vee \in \Sc(E^{n-1})$, 
we define the \emph{local Fourier--Jacobi period} by 
\begin{align*}
&\alpha(\varphi, \varphi^\vee, \varphi', \varphi'^\vee, \phi, \phi^\vee)
\\&= \int_{G_{n-1}}\int_{H_{n-1}} (\pi(hg)\varphi, \varphi^\vee)_{\pi} (\pi'(g)\varphi', \varphi'^\vee)_{\pi'}
\overline{(\omega_\psi(hg) \phi, \phi^\vee)} dhdg.
\end{align*} 

\begin{prop}\label{converge}
The integral $\alpha(\varphi, \varphi^\vee, \varphi', \varphi'^\vee, \phi, \phi^\vee)$
is absolutely convergent. 
\end{prop}
\begin{proof}
This is exactly the same as the symplectic-metaplectic case (\cite[Proposition 2.2.1]{X}). 
We omit the details. 
\end{proof}

Since the central character of $\omega_\psi$ is $\psi$, 
if $\alpha(\varphi, \varphi^\vee, \varphi', \varphi'^\vee, \phi, \phi^\vee) \not= 0$, 
then 
\[
\int_{F}(\pi(hg \cdot \bv(0,0;z)) \varphi, \varphi^\vee)_\pi \overline{\psi(z)} dz \not= 0
\]
for some $h \in H_{n-1}$ and $g \in G_{n-1}$.
This means that the image of $\varphi$ in $\pi_\psi$ is nonzero.
The converse holds in the following sense. 

\begin{lem}\label{vee}
Let $\varphi \in \pi$. 
Assume that the image of $\varphi$ in $\pi_\psi$ is nonzero.
Then there exists $\varphi^\vee \in \pi^\vee$ such that 
\[
\int_{F}(\pi(\bv(0,0;z)) \varphi, \varphi^\vee)_\pi \overline{\psi(z)} dz \not= 0.
\]
\end{lem}
\begin{proof}
Denote by $(\pi^\vee)^*$ the linear dual of $\pi^\vee$. 
Then $\pi$ is regarded as a subspace of $(\pi^\vee)^*$ via $(\cdot, \cdot)_\pi$.
By Proposition \ref{converge}, 
the map
\[
\pi \rightarrow (\pi^\vee)^*,\;
\varphi \mapsto 
\left[
\varphi^\vee \mapsto \int_{F}(\pi(\bv(0,0;z)) \varphi, \varphi^\vee)_\pi \overline{\psi(z)} dz
\right]
\]
is well-defined. 
It induces an injection $\pi_\psi \hookrightarrow (\pi^\vee)^*$.
Hence the assertion is proven.
\end{proof}

Now, we prove Theorem \ref{even} (1). 
\begin{proof}[Proof of Theorem \ref{even} (1)]
Let $\pi$ be an irreducible tempered representation of $G_n = \U_{2n}$. 
Suppose that $\pi_{\psi}^{K_{2m}^W} \not= 0$ for some $2m \geq 0$. 
We will show that $\pi$ must be $\psi_E$-generic.
\par

Fix $\varphi \in \pi^{K_{2m}^W}$ such that the image of $\pi_{\psi}$ is nonzero.
By Lemma \ref{vee}, one can find $\varphi^\vee \in \pi^\vee$ such that 
\[
\int_{F}(\pi(\bv(0,0;z)) \varphi, \varphi^\vee)_\pi \overline{\psi(z)} dz \not= 0.
\]
Since $Z$ is the center of $H_{n-1}$, 
we may assume that $\varphi^\vee$ is fixed by $K^H$.
Hence the matrix coefficient 
$H_{n-1} \ni h \mapsto (\pi(h) \varphi, \varphi^\vee)_\pi$ is bi-$K^H$-invariant. 
Since $\omega_{\psi}$ is the unique irreducible representation of $H_{n-1}$
whose central character is $\psi$, 
there are $\phi, \phi^\vee \in \Sc(E^{n-1})$ such that 
\[
\int_{H_{n-1}} (\pi(h) \varphi, \varphi^\vee)_\pi 
\overline{(\omega_{\psi}(h)\phi, \phi^\vee)} dh \not= 0.
\]
We may also assume that both $\phi$ and $\phi^\vee$ are fixed by $K^H$.
Since $\omega_{\psi}^{K^H} = \C \phi_0$, 
we can take $\phi = \phi^\vee = \phi_0$.
Hence
\[
\int_{H_{n-1}} (\pi(h) \varphi, \varphi^\vee)_\pi 
\overline{(\omega_{\psi}(h)\phi_0, \phi_0)} dh \not= 0.
\]
\par

Now by applying the same argument as \cite[Lemma 12.5]{GS} to the integral on $G_{n-1}$, 
one can find $\pi' \in \Irr_\temp(G_{n-1})$ and $(\varphi', \varphi'^\vee) \in \pi' \times \pi'^\vee$ such that 
\[
\alpha(\varphi, \varphi^\vee, \varphi', \varphi'^\vee, \phi_0, \phi_0) \not= 0.
\]
We may assume that $\varphi'$ is fixed by $K^{W_0}$ since so are $\varphi$ and $\phi_0$.
This means that $\pi'$ is unramified. 
By the local Gan--Gross--Prasad conjecture (\cite[Conjecture 17.3, Theorem 19.1]{GGP}), 
whose basic case is proven by Gan--Ichino \cite[Theorem 1.3]{GI}, 
we can deduce that $\pi$ is $\psi_E$-generic.
\par

Conversely, if $\pi$ is $\psi_E$-generic, 
by the local Gan--Gross--Prasad conjecture, 
one can find an irreducible tempered unramified representation $\pi'$ of $G_{n-1}$
such that $\Hom_{J_{n-1}}(\pi \otimes \pi' \otimes \overline{\omega_{\psi}}, \C) \not= 0$.
Since $\pi'$ and $\omega_\psi$ are irreducible as representations of $G_{n-1}$ and $H_{n-1}$, respectively,
for any nonzero unramified vector $\varphi'_0 \in \pi'$ and 
for any nonzero element $\LL \in \Hom_{J_{n-1}}(\pi \otimes \pi' \otimes \overline{\omega_{\psi}}, \C)$, 
one can take $\varphi \in \pi$ such that 
$\LL(\varphi \otimes \varphi_0' \otimes \overline{\phi_0}) \not= 0$.
We may assume that $\varphi$ is fixed by $K^{J}$. 
Since $\pi$ is smooth, $\varphi$ is fixed by $K_{2m}^{W}$ for $2m \gg 0$. 
In this case $\varphi$ gives a nonzero element in $\pi_\psi^{K_{2m}^W}$.
\par

This completes the proof of Theorem \ref{even} (1). 
\end{proof}

Recall in \cite[Corollary 16.3]{GGP} that 
for $\pi \in \Irr(G_n)$ and $\pi' \in \Irr(G_{n-1})$, we have
\[
\dim_\C \Hom_{J_{n-1}}(\pi \otimes \pi' \otimes \overline{\omega_\psi}, \C) \leq 1.
\]
It is worth to state the following result which was obtained by the above argument.

\begin{prop}\label{ggp}
Let $\pi$ be an irreducible tempered representation of $G_n$. 
Suppose that there is $\varphi \in \pi^{K_{2m}^W}$ whose image in $\pi_{\psi}$ is nonzero for some $2m \geq 0$.
Then there exist
\begin{itemize}
\item
an irreducible tempered unramified representation $\pi'$ of $G_{n-1}$; 
\item
an unramified vector $\varphi'_0 \in \pi'$; and 
\item
$\LL \in \Hom_{J_{n-1}}(\pi \otimes \pi' \otimes \overline{\omega_{\psi}}, \C)$
\end{itemize}
such that $\LL(\varphi \otimes \varphi'_0 \otimes \overline{\phi_0}) \not= 0$.
\end{prop}

\section{Uniqueness}\label{sec.unique}
In this section, we will prove Theorem \ref{even} (2). 
As usual, this is an application of Rankin--Selberg integrals.
\par

\subsection{Rankin--Selberg integrals}\label{sec.RS}
Let $\tau$ be an irreducible generic representation of $\GL_{n-1}(E)$
which is realized on the Whittaker space $\WW(\tau,\psi_E^{-1})$ with respect to 
the inverse of $\psi_E$. 
For $s \in \C$, we consider the normalized parabolically induced representation 
\[
\Ind_{Q_{n-1}}^{G_{n-1}}\left(\tau |\det|^{s-\half{1}}\right)
\]
of $G_{n-1}$, 
where $Q_{n-1} = M_{n-1}U_{n-1}$ denotes the standard Siegel parabolic subgroup so that 
\begin{align*}
M_{n-1} &= \{\bm(a) \;|\; a \in \GL_{n-1}(E)\}, \\
U_{n-1} &= \{\bn(b) \;|\; b \in \M_{n-1}(E), {}^t(w_{n-1}\overline{b}) = w_{n-1}b\}. 
\end{align*}
We realize it on the space $V_{Q_{n-1}}^{G_{n-1}}(\WW(\tau,\psi_E^{-1}), s)$
of smooth functions $f_s \colon G_{n-1} \times \GL_{n-1}(E) \rightarrow \C$
such that 
\begin{itemize}
\item
$f_s(\bn(b) \bm(a) g,a') = |\det a|_E^{s+\half{n}-1} f_s(g,a'a)$
for $g \in G_{n-1}$, $a,a' \in \GL_{n-1}(E)$ and $\bn(b) \in U_{n-1}$; 
\item
the function $a \mapsto f_s(g,a)$ belongs to $\WW(\tau,\psi_E^{-1})$ for any $g \in G_{n-1}$. 
\end{itemize}
Define a new representation $\tau^*$ by $\tau^*(a) = \tau(a^*)$, 
where $a^* = w_{n-1}{}^t\overline{a}^{-1}w_{n-1}^{-1}$. 
Note that $\tau^* \cong \overline{\tau}^\vee$, 
where $\overline{\tau}(a) = \tau(\overline{a})$.
As in \cite[Section 2.3.1]{Mori}, one can define a normalized intertwining operator 
\[
M^*(\tau,s) \colon V_{Q_{n-1}}^{G_{n-1}}(\WW(\tau,\psi_E^{-1}), s)
\rightarrow V_{Q_{n-1}}^{G_{n-1}}(\WW(\tau^*,\psi_E^{-1}), 1-s).
\]
\par

Let $\pi$ be an irreducible $\psi_E$-generic representation of $G_n$
realized on the Whittaker space $\WW(\pi,\psi_E)$. 
For $W \in \WW(\pi,\psi_E)$, $f_s \in V_{Q_{n-1}}^{G_{n-1}}(\WW(\tau,\psi_E^{-1}), s)$ 
and $\phi \in \Sc(E^{n-1})$, 
we define the \emph{Rankin--Selberg integral} $\LL(W,f_s,\overline{\phi})$ by 
\[
\int_{N_{n-1} \bs G_{n-1}} \int_{E^{n-1}}
W(w_{1,n-1} \bv(x,0;0)g ) f_s(g,\1_{n-1}) \overline{\omega_\psi(g) \phi(x)} dx dg, 
\]
where we set 
\[
w_{1,n-1} = \left(\begin{array}{cc|cc}
&\1_{n-1}&&\\
1&&&\\
\hline
&&&1\\
&&\1_{n-1}&
\end{array}\right)
\in G_n.
\]

\begin{rem}\label{W(w)}
Note that 
\[
W(w_{1,n-1} \bv(x,0;0)g \cdot \bv(0,0;z)) = \psi(z)W(w_{1,n-1}\bv(x,0;0)g)
\]
for $W \in \WW(\pi, \psi_E)$. 
Hence the restriction map $W \mapsto W(w_{1,n-1} \bv(x,0;0)g)$
factors through $\pi \twoheadrightarrow \pi_\psi$. 
In particular, if $\pi$ is $\psi_E$-generic, then $\pi_\psi$ is nonzero.
\end{rem}

\begin{thm}\label{RS}
Keep the notations.
\begin{enumerate}
\item
The integral $\LL(W,f_s,\overline{\phi})$ converges absolutely for $\re(s) \gg 0$.
It is a rational function in $q^{-s}$ 
so that it admits a meromorphic continuation to the whole $s$-plane.

\item 
Let $I(\pi \times \tau \times \chi)$ be the fractional ideal of $\C[q^{-s}, q^{s}]$
generated by $\LL(W,f_s,\overline{\phi})$ 
for $W \in \WW(\pi,\psi_E)$, $f_s \in V_{Q_{n-1}}^{G_{n-1}}(\WW(\tau,\psi_E^{-1}), s)$ 
and $\phi \in \Sc(E^{n-1})$. 
Then there is a unique polynomial $P(X) \in \C[X]$ with $P(0)=1$
such that $I(\pi \times \tau \times \chi) = (P(q^{-s})^{-1})$. 
We define the \emph{$L$-function} attached to $\pi \times \tau$ and $\chi$ by 
\[
L(s, \pi \times \tau, \chi) = P(q^{-s})^{-1}.
\]

\item
There is a meromorphic function $\Gamma(s, \pi \times \tau, \psi)$
such that 
\[
\LL(W, M^*(\tau, s)f_s, \overline{\phi}) = 
\omega_\pi(-1)^{n-1} \omega_\tau(-1)^n
\Gamma(s, \pi \times \tau, \chi, \psi)
\LL(W,f_s,\overline{\phi})
\]
for any $W \in \WW(\pi,\psi_E)$, $f_s \in V_{Q_{n-1}}^{G_{n-1}}(\WW(\tau,\psi_E^{-1}), s)$ 
and $\phi \in \Sc(E^{n-1})$.
We call $\Gamma(s, \pi \times \tau, \chi, \psi)$ the \emph{gamma factor}
attached to $\pi \times \tau$, $\chi$ and $\psi$. 

\item
The gamma factor $\Gamma(s, \pi \times \tau, \chi, \psi)$ satisfies several properties 
(including the multiplicativity), 
which determine $\Gamma(s, \pi \times \tau, \chi, \psi)$ uniquely. 

\item
Define the $\ep$-factor attached to $\pi \times \tau$, $\chi$ and $\psi$ 
by 
\[
\ep(s, \pi \times \tau, \chi, \psi) 
= \Gamma(s, \pi \times \tau, \chi, \psi) 
\frac{L(s, \pi \times \tau, \chi)}{L(1-s, \pi^\vee \times \tau^\vee, \chi)}.
\]
Then it satisfies that 
\[
\ep(1-s, \pi \times \tau^*, \chi, \psi)\ep(s, \pi \times \tau, \chi, \psi) = 1.
\]
In particular, $\ep(s,\pi \times \tau, \chi, \psi) \in \C^\times (q^{-s})^\Z$.

\end{enumerate}
\end{thm}
\begin{proof}
(1) is \cite[Proposition 6.4]{BS}. 
By \cite[Proposition 6.5]{BS}, we see that $1 \in I(\pi \times \tau \times \chi)$, 
which implies (2). 
The assertion (3) follows from the multiplicity one theorem proven in \cite[Corollary 16.3]{GGP}. 
(4) is proven by Morimoto \cite[Theorem 3.1]{Mori}.
Since $M^*(\tau^*, 1-s) \circ M^*(\tau, s) = \id$, 
using $\omega_{\tau^*}(-1) = \omega_{\tau}(-1)$, 
we have 
\begin{align*}
\LL(W, f_s, \overline{\phi}) &= \LL(W, M^*(\tau^*, 1-s) \circ M^*(\tau, s)f_s, \overline{\phi}) 
\\&= \omega_\pi(-1)^{n-1}\omega_{\tau^*}(-1)^n
\Gamma(1-s, \pi \times \tau^*, \chi, \psi) \LL(W, M^*(\tau, s)f_s, \overline{\phi}) 
\\&= 
\Gamma(1-s, \pi \times \tau^*, \chi, \psi)\Gamma(s, \pi \times \tau, \chi, \psi) \LL(W, f_s, \overline{\phi}) 
\end{align*}
for any $W$, $f_s$ and $\phi$.
It means that 
\[
\Gamma(1-s, \pi \times \tau^*, \chi, \psi)\Gamma(s, \pi \times \tau, \chi, \psi) = 1, 
\]
which is equivalent to saying that
\[
\ep(1-s, \pi \times \tau^*, \chi, \psi)\ep(s, \pi \times \tau, \chi, \psi) = 1.
\]
Hence $\ep(s,\pi \times \tau, \chi, \psi) \in \C[q^{-s}, q^{s}]^\times = \C^\times (q^{-s})^\Z$.
\end{proof}

\subsection{Unramified representations}\label{sec.unram}
In this subsection, we consider the Rankin--Selberg integrals 
when $\tau$ varies over irreducible unramified representations of $\GL_{n-1}(E)$. 
\par

Recall that $K^{W_0} = K_{0}^W \cap G_{n-1}$. 
It is a hyperspecial maximal compact subgroup of $G_{n-1}$, 
and the Iwasawa decomposition $G_{n-1} = Q_{n-1} K^{W_0}$ holds. 
\par

Irreducible unramified representations of $\GL_{n-1}(E)$
are parametrized by the \emph{Satake parameters} 
$\ub{x} = (x_1,\dots,x_{n-1}) \in (\C^\times)^{n-1}/S_{n-1}$. 
We write the unramified representation associated to $\ub{x}$ by $\tau_{\ub{x}}$. 
Then for almost all $\ub{x}$, 
since $\tau_{\ub{x}}$ is generic, 
there exists a unique function 
$f_{s}(\ub{x}) \in V_{Q_{n-1}}^{G_{n-1}}(\WW(\tau,\psi_E^{-1}), s)$
such that 
\begin{itemize}
\item
$f_{s}(gk, a; \ub{x}) = f_{s}(g, a; \ub{x})$ for any $g \in G_{n-1}$, $k \in K^{W_0}$
and $a \in \GL_{n-1}(E)$; and
\item
the function $W(a; \ub{x}) = f_{s}(\1_{2(n-1)}, a; \ub{x})$ 
is right $\GL_{n-1}(\oo_E)$-invariant with $W(\1_{n-1}; \ub{x}) = 1$. 
\end{itemize}
\par

\begin{lem}\label{intertwining}
For $\ub{x} = (x_1, \dots, x_{n-1})$, we write $\ub{x}^{-1} = (x_1^{-1}, \dots, x_{n-1}^{-1})$. 
Then we have 
\begin{align*}
&\frac{M^*(\tau_{\ub{x}}, s)f_{s}(\ub{x})}
{\prod_{i=1}^{n-1}(1-q^{-s}x_i)\prod_{1 \leq i < j \leq n-1}(1-q^{-2s}x_ix_j)}
\\&=
\frac{f_{1-s}(\ub{x}^{-1})}
{\prod_{i=1}^{n-1}(1-q^{-(1-s)}x_i^{-1})\prod_{1 \leq i < j \leq n-1}(1-q^{-2(1-s)}x_i^{-1}x_j^{-1})}
\end{align*}
\end{lem}
\begin{proof}
The assertion follows from \cite[Theorem 8.1]{BS} and \cite[Theorem 3.1 (c)]{Mori}.
\end{proof}

Let $\pi$ be an irreducible $\psi_E$-generic tempered representation of $G_n$ with $L$-parameter $\phi_\pi$. 
Then by the uniqueness of the gamma factor (Theorem \ref{RS} (4)), 
we have 
\[
\Gamma(s, \pi \times \tau_{\ub{x}}, \chi, \psi) = 
\prod_{i=1}^{n-1} 
\ep(s+s_i+s_0, \phi_\pi, \psi_E)\frac{L(1-s-s_i-s_0, \phi_\pi^\vee)}{L(s+s_i+s_0, \phi_\pi)}
\]
for almost all $\ub{x} = (x_1,\dots,x_{n-1})$, 
where $s_0, s_1, \dots, s_{n-1} \in \C$ are so that 
$q^{-2s_0} = -1$ and $x_i = q^{-2s_i}$ for $1 \leq i \leq n-1$.
Since $\phi_\pi$ is tempered, two meromorphic functions
$\prod_{i=1}^{n-1}L(1-s-s_i-s_0, \phi_\pi^\vee)$ and $\prod_{i=1}^{n-1}L(s+s_i+s_0, \phi_\pi)$ 
have no common pole for almost all $\ub{x}$. 
In particular, in this case, we have
\begin{align*}
L(s, \pi \times \tau_{\ub{x}}, \chi) &= \prod_{i=1}^{n-1}L(s+s_i+s_0, \phi_\pi), \\
\ep(s, \pi \times \tau_{\ub{x}}, \chi, \psi) &= \prod_{i=1}^{n-1} \ep(s+s_i+s_0, \phi_\pi, \psi_E). 
\end{align*}
If we write $L(s, \phi_\pi) = P_\pi(q^{-2s})$ 
and $\ep(s, \phi_\pi, \psi_E) = \ep q^{c(\phi_\pi)(1-2s)}$, 
then 
\begin{align*}
L(s, \pi \times \tau_{\ub{x}}, \chi) &= \prod_{i=1}^{n-1}P_\pi(-x_iq^{-2s}), \\
\ep(s, \pi \times \tau_{\ub{x}}, \chi, \psi) 
&= \ep^{n-1} (-q^{1-2s})^{c(\phi_\pi)(n-1)} \prod_{i=1}^{n-1}x_i^{c(\phi_\pi)}.
\end{align*}

\subsection{Proof of Theorem \ref{even} (2)}
The symmetric group $S_{n-1}$ acts on $\C[X_1^{\pm1},\dots, X_{n-1}^{\pm1}]$ canonically.
Set
\[
\TT = \C[X_1^{\pm1},\dots,X_{n-1}^{\pm1}]^{S_{n-1}}.
\]
Note that 
\[
\TT = \C[T_1,\dots,T_{n-2}, T_{n-1}, T_{n-1}^{-1}]
\]
with 
\[
T_i = \sum_{\sigma \in S_{n-1}} X_{\sigma(1)} \cdots X_{\sigma(i)}.
\]
The degree with respect to $T_{n-1}$ gives a $\Z$-grading on $\TT$, 
i.e., $\TT = \oplus_{d \in \Z} \TT_d$ with 
\[
\TT_d = \C[T_1,\dots,T_{n-2}] T_{n-1}^d.
\]
\par

Write $\ub{X} = (X_1,\dots,X_{n-1})$ and $q^{1-2s}\ub{X} = (q^{1-2s}X_1, \dots, q^{1-2s}X_{n-1})$. 
There is a function 
\[
W(\ub{X}) \colon \GL_{n-1}(E) \rightarrow \TT
\]
such that $W(\ub{X})|_{\ub{X} = \ub{x}} = W(\ub{x})$
for almost all $\ub{x} \in (\C^\times)^{n-1}$.
Similarly, we consider the function 
$f_{s}(\ub{X}) \colon G_{n-1} \times \GL_{n-1}(E) \rightarrow \TT$ 
so that $f_{s}(\ub{X})|_{\ub{X} = \ub{x}} = f_{s}(\ub{x})$
for almost all $\ub{x} \in (\C^\times)^{n-1}$.
In particular, $f_{s}(\1_{2(n-1)}, a; \ub{X}) = W(a; q^{1-2s}\ub{X})$. 
\par

We regard $\LL(W, f_{1/2}(\ub{X}), \overline{\phi_0})$ 
as a formal power series of $X_1^{\pm1}, \dots, X_{n-1}^{\pm1}$, 
or an element of $\C[T_1,\dots,T_{n-2}] [[T_{n-1}^{\pm1}]]$.
For $\lambda = (\lambda_1,\dots, \lambda_{n-1}) \in \Z^{n-1}$, 
we set $|\lambda| = \lambda_1+\dots+\lambda_{n-1}$. 
The following is a key lemma.

\begin{lem}\label{degree}
Let $W \in \WW(\pi, \psi_E)^{K_{2m}^W}$. 
Write 
\begin{align*}
\LL(W, f_{1/2}(\ub{X}), \overline{\phi_0}) 
&= 
\sum_{\lambda \in \Z^{n-1}} a_\lambda(W) X_1^{\lambda_1}\cdots X_{n-1}^{\lambda_{n-1}} 
=
\sum_{d \in \Z} \LL_d(W) T_{n-1}^d
\end{align*}
with $a_\lambda(W) \in \C$ and $\LL_d(W) \in \C[T_1, \dots, T_{n-2}]$.
Then 
\begin{itemize}
\item
$a_\lambda(W) = 0$ unless $|\lambda| \geq -(n-1)m$; and
\item
$\LL_d(W) = 0$ unless $d \geq -m$.
\end{itemize}
\end{lem}
\begin{proof}
For row vectors $x, u \in E^{n-1}$ and $a \in \GL_{n-1}(E)$, 
we put $k(x,a,u)$ to be the matrix
\[
(w_{1,n-1}\bv(x,0;0)\bm(a))^{-1}
\left(
\begin{array}{cc|cc}
\1_{n-1} & {}^tu & 0 & 0 \\
0 & 1 & 0 & 0 \\
\hline
0 & 0 & 1 & -\overline{u}w_{n-1} \\
0 & 0 & 0 & \1_{n-1}
\end{array}
\right)
w_{1,n-1}\bv(x,0;0)\bm(a). 
\]
By an easy calculation, $k(x,a,u)$ is equal to
\[
\left(
\begin{array}{cc|cc}
1-x{}^tu & -x{}^tuxa & 0 & 0 \\
a^{-1}{}^tu & \1_{n-1} + a^{-1}{}^tuxa & 0 & 0 \\
\hline
0 & 0 & \1_{n-1} -w_{n-1}{}^t\overline{a}{}^t\overline{x} \overline{u} {}^t\overline{a}^{-1}w_{n-1}
& w_{n-1}{}^t\overline{a}{}^t\overline{x} \overline{u}{}^t\overline{x}\\
0 & 0 & -\overline{u}{}^t\overline{a}^{-1}w_{n-1} & 1+\overline{u}{}^t\overline{x}
\end{array}
\right).
\]
In particular, 
if $xa \in \oo_E^{n-1}$ and $u{}^ta^{-1} \in (\pp_E^{m})^{n-1}$, 
then $x{}^tu \in \pp_E^{m}$ so that $k(x,a,u) \in K_{2m}^W$.
\par

As functions on $g \in G_{n-1}$, 
all of $W(w_{1,n-1}\bv(x,0;0)g)$, $f_{s}(g,\1_{n-1}; \ub{X})$
and $\overline{\omega_{\psi}(g)\phi_0}$
are right $K^{W_0}$-invariant. 
Hence, by the integral formula with respect to the Iwasawa decomposition, 
we can write $\LL(W,f_{s}(\ub{X}), \overline{\phi_0})$ as
\[
\int_{T_{n-1}}\int_{E^{n-1}}
W(w_{1,n-1}\bv(x,0;0)t) 
f_{s}(t, \1_{n-1}; \ub{X}) 
\overline{\omega_{\psi}(t)\phi_0(x)} 
\delta_{B_{n-1}}^{-1}(t)dxdt, 
\]
where $B_{n-1} = T_{n-1}N_{n-1}$ is the upper triangular Borel subgroup of $G_{n-1}$
with the diagonal torus $T_{n-1}$. 
Write $t = \bm(a)$ with $a = \diag(a_1,\dots,a_{n-1})$ being a diagonal matrix in $\GL_{n-1}(E)$.
Then $\omega_{\psi}(\bm(a))\phi_0(x) \not= 0 \iff xa \in \oo_E^{n-1}$.
In this case, if $W(w_{1,n-1}\bv(x,0;0)\bm(a)) \not= 0$, 
then for $u = (u_1,\dots,u_{n-1}) \in E^{n-1}$ such that $u{}^ta^{-1} \in (\pp_E^{m})^{n-1}$, 
we have
\begin{align*}
0 &\not= W(w_{1,n-1}\bv(x,0;0)\bm(a)) 
\\&= W(w_{1,n-1}\bv(x,0;0)\bm(a) \cdot k(x,a,u)) 
\\&= W\left(
\left(
\begin{array}{cc|cc}
\1_{n-1} & {}^tu & 0 & 0 \\
0 & 1 & 0 & 0 \\
\hline
0 & 0 & 1 & -\overline{u}w_{n-1} \\
0 & 0 & 0 & \1_{n-1}
\end{array}
\right)
w_{1,n-1}\bv(x,0;0)\bm(a)
\right)
\\&= \psi_E(u_{n-1})W(w_{1,n-1}\bv(x,0;0)\bm(a)).
\end{align*}
This shows that
\[
u_{n-1} \in \pp_E^{\ord(a_{n-1})+m}
\implies \psi_E(u_{n-1}) = 1.
\]
This means that $\ord(a_{n-1})+m \geq 0$.
\par

Recall that 
$f_{s}(\bm(a), \1_{n-1}; \ub{X}) = \delta_{Q_{n-1}}^{\half{1}}(\bm(a)) W(a; q^{1-2s}\ub{X})$.
By a similar (and well-known) argument, 
if $W(a; q^{1-2s}\ub{X}) \not= 0$, 
then $\ord(a_1) \geq \dots \geq \ord(a_{n-1})$. 
Hence we conclude that 
if 
\[
W(w_{1,n-1}\bv(x,0;0)\bm(a)) 
W(a; \ub{X}) \overline{\omega_{\psi}(\bm(a))\phi_0(x)} \not= 0,
\] 
then 
\[
\ord(a_1) \geq \dots \geq \ord(a_{n-1}) \geq -m
\]
so that 
\[
\ord(\det(a)) = \sum_{i=1}^{n-1}\ord(a_i) \geq -(n-1)m. 
\] 
Since the Casselman--Shalika formula \cite{CS} tells us that
\[
W(a; \ub{X}) 
\in 
\left(
\bigoplus_{\substack{\lambda \in \Z^{n-1} \\ |\lambda| = \ord(\det(a))}}
\C X_1^{\lambda_1}\cdots X_{n-1}^{\lambda_{n-1}} 
\right)
\cap \C[T_1,\dots,T_{n-2}]T_{n-1}^{\ord(a_{n-1})}, 
\]
we obtain the assertions.
\end{proof}

For $W \in \WW(\pi, \psi_E)^{K_{2m}^W}$, 
we define $\Psi(W; \ub{X})$ by 
\[
\Psi(W; \ub{X})
= 
\frac{\prod_{i=1}^{n-1}P_\pi(-q^{-1}X_i)\LL(W, f_{1/2}(\ub{X}), \phi_0)}
{\prod_{i=1}^{n-1}(1-q^{-1}X_i)\prod_{1\leq i < j \leq n-1}(1-q^{-2}X_iX_j)}.
\]

\begin{prop}\label{Psi}
If $2m < c(\phi_\pi)$, then $\Psi(W; \ub{X}) = 0$ for $W \in \WW(\pi, \psi_E)^{K_{2m}^W}$. 
If $2m = c(\phi_\pi)$ or $2m = c(\phi_\pi)+1$, 
then
\[
\dim_\C \left\{ \Psi(W; \ub{X}) \;\middle|\; 
W \in \WW(\pi, \psi_E)^{K_{2m}^W}\right\} 
\leq 1.
\]
\end{prop}
\begin{proof}
Since $P_\pi(X)$ is a polynomial of $X$ with $P_\pi(0) = 1$, 
and since $(1-q^{-1}X_i)^{-1} = \sum_{k=0}^\infty (q^{-1}X_i)^k$
and $(1-q^{-2}X_iX_j)^{-1} = \sum_{k=0}^\infty(q^{-2}X_iX_j)^k$, 
if we write 
\begin{align*}
\Psi(W; \ub{X}) 
&= 
\sum_{\lambda \in \Z^{n-1}} \alpha_\lambda(W) X_1^{\lambda_1}\cdots X_{n-1}^{\lambda_{n-1}} 
=
\sum_{d \in \Z} \Psi_d(W; \ub{X}) T_{n-1}^d
\end{align*}
with $\alpha_\lambda(W) \in \C$ and $\Psi_d(W; \ub{X}) \in \C[T_1, \dots, T_{n-2}]$, 
by Lemma \ref{degree}, 
we see that 
\begin{itemize}
\item
$\alpha_\lambda(W) = 0$ unless $|\lambda| \geq -(n-1)m$; and

\item
$\Psi_d(W; \ub{X}) = 0$ unless $d \geq -m$.
\end{itemize}
\par

Write $\ub{X}^{-1} = (X_1^{-1}, \dots, X_{n-1}^{-1})$. 
By the functional equation (Theorem \ref{RS} (3), (5))
together with Lemma \ref{intertwining}, 
we see that 
\[
\tag{$\ast$}
T_{n-1}^{-c(\phi_\pi)} \Psi(W; \ub{X}^{-1}) = \ep_0 \Psi(W; \ub{X})
\]
with 
\[
\ep_0 = ((-1)^{c(\phi_\pi)} \ep \cdot \omega_\pi(-1))^{n-1}. 
\]
The left hand side and the right hand side of $(\ast)$
belong to 
\[
\bigoplus_{d \leq m-c(\phi_\pi)} \C[T_1, \dots, T_{n-2}] T_{n-1}^{d}, 
\quad
\bigoplus_{d \geq -m} \C[T_1, \dots, T_{n-2}] T_{n-1}^{d}, 
\]
respectively. 
Hence if $\Psi_d(W; \ub{X}) \not= 0$, 
then $-m \leq d \leq m-c(\phi_\pi)$ so that $2m \geq c(\phi_\pi)$.
A similar argument shows that
if $\alpha_{\lambda}(W) \not= 0$, then 
\[
-(n-1)m \leq |\lambda| \leq (n-1)(m-c(\phi_\pi)).
\]
\par

Now we assume that $2m = c(\phi_\pi)$. 
Then $\Psi_d(W; \ub{X}) = 0$ unless $d = -m$. 
Hence
\[
T_{n-1}^{m} \Psi(W; \ub{X}) \in \C[T_1, \dots, T_{n-2}] \subset \C[X_1, \dots, X_{n-1}].
\]
This implies that $\alpha_{\lambda}(W) = 0$ 
unless $\lambda_i \geq -m$ for any $1 \leq i \leq n-1$.
On the other hand, since $\alpha_{\lambda}(W) = 0$ unless $|\lambda| = -(n-1)m$, 
we see that $\alpha_{\lambda}(W) = 0$ unless $\lambda_1 = \dots = \lambda_{n-1} = -m$. 
This means that 
\[
\Psi(W; \ub{X}) \in \C T_{n-1}^{-m}
\]
so that 
\[
\dim_\C \left\{ \Psi(W; \ub{X}) \;\middle|\; 
W \in \WW(\pi, \psi_E)^{K_{c(\phi_\pi)}^W}\right\} 
\leq 1.
\]
\par

Next we assume that $2m = c(\phi_\pi)+1$.
Then $\Psi_d(W; \ub{X}) = 0$ unless $d = -m, -m+1$, 
and $\alpha_\lambda(W) = 0$ unless $|\lambda| = -(n-1)m, -(n-1)(m-1)$.
In particular, $\Psi_{-m+1}(W; \ub{X})$ is a scalar so that
\[
\Psi_{-m+1}(W; \ub{X}^{-1}) = \Psi_{-m+1}(W; \ub{X}).
\]
By the functional equation $(\ast)$, we have 
\begin{align*}
\Psi_{-m+1}(W; \ub{X}^{-1}) &= \ep_0 \Psi_{-m}(W; \ub{X}), \\
\Psi_{-m}(W; \ub{X}^{-1}) &= \ep_0 \Psi_{-m+1}(W; \ub{X}).
\end{align*}
Hence $\Psi_{-m}(W; \ub{X})$ is also a scalar.
Therefore, 
\[
\Psi(W; \ub{X}) \in \C(T_{n-1}^{-m} + \ep_0 T_{n-1}^{-m+1})
\]
so that 
\[
\dim_\C \left\{ \Psi(W; \ub{X}) \;\middle|\; 
W \in \WW(\pi, \psi_E)^{K_{c(\phi_\pi)+1}^W}\right\} 
\leq 1.
\]
This completes the proof.
\end{proof}

By Proposition \ref{ggp}, 
we see that $\WW(\pi, \psi_E)^{K_{2m}^W} \ni W \mapsto \Psi(W; \ub{X})$ gives 
an injective linear map
\[
\Psi \colon \pi_{\psi}^{K_{2m}^W} \hookrightarrow \TT.
\]
Hence by Proposition \ref{Psi}, we have
\begin{itemize}
\item
$\pi_{\psi}^{K_{2m}^W} = 0$ if $2m < c(\phi_\pi)$; and 
\item
$\dim_\C(\pi_{\psi}^{K_{2m}^W}) \leq 1$ if $2m = c(\phi_\pi)$ or $2m = c(\phi_\pi)+1$.
\end{itemize}
This completes the proof of Theorem \ref{even} (2).

\section{Existence}\label{sec.existence}
In this section, we will prove Theorem \ref{even} (3). 
To do this, we will use the theta correspondence for $(\U(V), \U(W))$. 

\subsection{Theta correspondence}
Recall that $V = V_{2n+1}$ (\resp $W = W_{2n}$) is 
a hermitian (\resp skew-hermitian) space over $E$ of dimension $2n+1$ (\resp $2n$). 
Then $\W = V \otimes_E W$ forms a symplectic space of dimension $4n(2n+1)$
equipped with the symplectic form
\[
\pair{v \otimes w, v' \otimes w'} = \tr_{E/F}\left( \pair{v,v'}_V \cdot \pair{w,w'}_W \right).
\]
Here, $\U(V)$, $\U(W)$ and $\Sp(\W)$ act on $V$, $W$ and $\W$, respectively, 
all from the left. 
We have a canonical map 
$\U(V) \times \U(W) \rightarrow \Sp(\W)$. 
\par

Recall that $\chi$ is the unique non-trivial quadratic unramified character of $E^\times$.
Note that $\chi|_{F^\times}$ is equal to the quadratic character corresponding to $E/F$.
Let $\tl\Sp(\W)$ be the metaplectic $\C^\times$-cover.
Using the pair $(\chi_V, \chi_W) = (\chi^{2n+1}, \chi^{2n})$, 
we have Kudla's splitting \cite{Kudla}
\[
\U(V) \times \U(W) \rightarrow \tl\Sp(\W).
\]
\par

Let $\omega_\psi$ be the Weil representation of $\tl\Sp(\W)$ associated to the additive character $\psi$.
By the pullback, we obtain the Weil representation $\omega_{\psi, V, W}$ of $\U(V) \times \U(W)$.
For an irreducible representation $\pi$ of $\U(W)$, 
it is known that the maximal $\pi$-isotypic quotient of $\omega_{\psi, V, W}$ is of the form
\[
\Theta_\psi(\pi) \boxtimes \pi
\]
for a smooth representation $\Theta_\psi(\pi)$ of $\U(V)$ of finite length.
The Howe duality conjecture, proven by Waldspurger \cite{W}, asserts that 
if $\Theta_\psi(\pi)$ is nonzero, then 
it has a unique irreducible quotient $\theta_\psi(\pi)$.
We call $\theta_\psi(\pi)$ the \emph{theta lift} of $\pi$.
\par

The following is a special case of Prasad's conjecture, 
which was proven by Gan--Ichino \cite{GI}.
See also Theorem 4.4 in that paper.

\begin{thm}
Let $\pi$ be an irreducible $\psi_E$-generic representation of $\U(W)$ with $L$-parameter $\phi_\pi$. 
Then $\Theta_\psi(\pi)$ is always nonzero.
Moreover, $\sigma = \theta_\psi(\pi)$ is generic and its $L$-parameter is given by 
\[
\phi_\sigma = \phi_\pi\chi \oplus \1, 
\]
where $\phi_\pi \chi = \phi_\pi \otimes \chi$.
\end{thm}

In particular, if $\sigma = \theta_\psi(\pi)$, then 
we have $c(\phi_\sigma) = c(\phi_\pi)$ and $\omega_{\sigma} = \omega_{\pi}$. 
Moreover, if $\pi$ is tempered, 
then so is $\sigma$ so that we have $\sigma^{K_{2m}^V} \not= 0$
for $2m = c(\phi_\pi)$ or $2m = c(\phi_\pi)+1$ by Theorem \ref{odd}.

\subsection{Lattice model}
First, we will show that $\pi^{K_{2m}^W} \not= 0$.
To do this, we use a lattice model $\Sc = \Sc(A)$ of 
the Weil representation $\omega_\psi$ of $\tl\Sp(\W)$.
In this subsection, we recall this model.
\par

Let $\W$ be a symplectic space over $F$ of dimension $2N$ 
equipped with a symplectic form $\pair{\cdot, \cdot}$.
The group law of the Heisenberg group $H(\W) = \W \oplus F$ is given by 
\[
(w_1,t_1) \cdot (w_2, t_2) = \left(w_1+w_2, t_1+t_2+\half{1}\pair{w_1,w_2}\right), 
\]
whose center is $\{0\} \oplus F \cong F$.
By the Stone--von Neumann theorem, 
there is a unique (up to isomorphism) irreducible admissible representation $(\rho_\psi, \Sc)$ of $H(\W)$
whose central character is $\psi$.
The symplectic group $\Sp(\W)$ acts on $H(\W)$ by $g \cdot (w,t) = (gw,t)$.
By the uniqueness, for $g \in \Sp(\W)$, we have $M_g \in \mathrm{Aut}(\Sc)$ such that
\[
\tag{\text{$\star$}}\label{star}
M_g \circ \rho_\psi(h) \circ M_g^{-1} = \rho_\psi(gh)
\quad\text{for $h \in H(\W)$}.
\]
By Schur's lemma, such $M_g$ is determined uniquely up to a nonzero scalar.
Define the metaplectic $\C^\times$-cover $\tl{\Sp}(\W)$ of $\Sp(\W)$ by
\[
\tl{\Sp}(\W) = \{(g,M_g) \in \Sp(\W) \times \mathrm{Aut}(\Sc) \;|\; 
\text{$M_g$ satisfies $(\ref{star})$}\}.
\]
We have an exact sequence
\[
\begin{CD}
1 @>>> \C^\times @>\alpha>> \tl{\Sp}(\W) @>\beta>> \Sp(\W) @>>> 1
\end{CD}
\]
given by $\alpha(z) = (\1_\W, z \cdot \id_\Sc)$ and $\beta(g,M_g) = g$.
The \emph{Weil representation} $\omega_\psi$ of $\tl{\Sp}(\W)$ on the space $\Sc$ is defined by 
\[
\omega_\psi(g,M_g) = M_g.
\]
\par

Now we shall give a realization of the space $\Sc$.
Let $A$ be a lattice of $\W$, i.e., a free $\oo_F$-submodule of rank $2N$.
The \emph{dual lattice} $A^*$ is defined by 
\[
A^* = \left\{w \in \W \;\middle|\; \text{$\pair{w,a} \in \oo_F$ for any $a \in A$} \right\}.
\]
Suppose that $A$ is self-dual, i.e., $A^* = A$.
Let $\Sc(A)$ be the space of locally constant, compactly supported functions 
$\phi \colon H(\W) \rightarrow \C$ such that
\[
\phi((a,t) \cdot h) = \psi(t) \phi(h)
\]
for $(a,t) \in A \oplus F$ and $h \in H(\W)$.
The group $H(\W)$ acts on $\Sc(A)$ by the right translation $\rho_\psi$.
It is known that 
the representation $(\rho_\psi, \Sc(A))$ of $H(\W)$ is irreducible with the central character $\psi$.
This gives a realization $(\omega_\psi, \Sc(A))$ of the Weil representation 
which is called a \emph{lattice model}.
Since $(a,0) \cdot (w,0) = (a+w, \half{1}\pair{a,w})$, 
by the restriction to $\W \oplus \{0\}$, 
we can identify $\Sc(A)$ with the space of locally constant, compactly supported functions 
$\phi \colon \W \rightarrow \C$ such that
\[
\phi(a+w) = \psi\left(-\half{1}\pair{a,w}\right) \phi(w)
\]
for $a \in A$ and $w \in \W$.
\par

For $g \in \Sp(\W)$, we define $M[g] \in \mathrm{Aut}(\Sc(A))$ by 
\[
(M[g]\phi)(w) = \int_{A} \psi\left(\half{1}\pair{a,w}\right) \phi(g^{-1} \cdot (a+w)) da
\]
for $\phi \in \Sc(A)$ and $w \in \W$.
Here, $da$ is the Haar measure on $A$ normalized so that $\vol(A) = 1$.
It is easy to check that $(g,M[g]) \in \tl{\Sp}(\W)$.
\par

Let $K_A$ be the stabilizer of $A$ in $\Sp(\W)$. 
Then we have
\[
(M[k]\phi)(w) = \phi(k^{-1} \cdot w)
\]
for $k \in K_A$, $\phi \in \Sc(A)$, and $w \in \W$.
The map $k \mapsto (k,M[k])$ gives a splitting $K_A \rightarrow \tl{\Sp}(\W)$.
If we identify $K_A$ with the image, 
the restriction of the Weil representation $(\omega_\psi, \Sc(A))$ to $K_A$
is given by $\omega_\psi(k)\phi(w) = \phi(k^{-1} \cdot w)$.
\par

\subsection{Families of lattices}
Take bases $\{e_n, \dots, e_1, e_0, e_{-1}, \dots, e_{-n}\}$ of $V$
and $\{f_n,\dots,f_1, f_{-1},\dots,f_{-n}\}$ of $W$, respectively, as in \S \ref{sec.gp}.
Set 
\begin{align*}
\Gamma_V &= 
\left(\bigoplus_{i=1}^n \oo_E e_i\right) \oplus \oo_E e_0 \oplus \left(\bigoplus_{i=1}^n \oo_E e_{-i} \right), \\
\Gamma_W &= \left(\bigoplus_{i=1}^n \oo_E f_i\right) \oplus \left(\bigoplus_{i=1}^n \oo_E f_{-i} \right).
\end{align*}
Then $\Gamma_V$ and $\Gamma_W$ are self-dual lattices, i.e.,
$\Gamma_V^* = \Gamma_V$ and $\Gamma_W^* = \Gamma_W$. 
\par

In this subsection,
for two $\oo_E$-modules $\Gamma_1$ and $\Gamma_2$, 
we denote by $\Gamma_1 \otimes \Gamma_2$ the tensor product of $\oo_E$-modules.
We put 
\[
A = \Gamma_V \otimes \Gamma_W.
\]
This is a self-dual lattice of $\W = V \otimes_F W$, i.e. $A^* = A$.
We will consider the lattice model $(\omega_\psi, \Sc(A))$ of the Weil representation of $\tl\Sp(\W)$.
\par

Fix a non-negative even integer $2m \geq 0$.
We consider lattices
\begin{align*}
M_{2m} &= 
\left(\bigoplus_{i=1}^n \oo_E e_i\right) 
\oplus \pp_E^{m} e_0 
\oplus \left(\bigoplus_{i=1}^n \oo_E e_{-i}\right),\\
N_{2m} &= 
\left(\bigoplus_{i=1}^n \oo_E f_i \right) 
\oplus \left(\bigoplus_{i=1}^{n-1} \oo_E f_{-i} \right) \oplus \pp_E^{m} f_{-n}
\end{align*}
of $V$ and $W$, respectively. 
Then $M_{2m} \subset \Gamma_V$ and $N_{2m} \subset \Gamma_W$.
Moreover, the dual lattices are given by
\begin{align*}
M_{2m}^* &= 
\left(\bigoplus_{i=1}^n \oo_E e_i\right) 
\oplus \pp_E^{-m} e_0 
\oplus \left(\bigoplus_{i=1}^n \oo_E e_{-i}\right), \\
N_{2m}^* &= 
\pp_E^{-m} f_n \oplus \left(\bigoplus_{i=1}^{n-1} \oo_E f_i \right) 
\oplus \left(\bigoplus_{i=1}^{n} \oo_E f_{-i} \right).
\end{align*}
\par

Recall that 
in Section \ref{sec.cpt}, 
we defined compact subgroups $K_{2m}^V$ and $K_{2m}^W$ of $\U(V)$ and $\U(W)$, respectively. 
The following lemma is easy to check. 
\begin{lem}
We have
\begin{align*}
K_{2m}^V = \{h \in \U(V) \;|\; (h-1) \cdot M_{2m}^* \subset M_{2m} \}, \\
K_{2m}^W = \{g \in \U(W) \;|\; (g-1) \cdot N_{2m}^* \subset N_{2m} \}. 
\end{align*}
In particular, $K_{2m}^V \times K_{2m}^W$ is contained in $K_A$
under the canonical map $\U(V) \times \U(W) \rightarrow \Sp(\W)$.
\end{lem}

Let $\Sc(A)_{M_{2m}}$ be the subspace of $\Sc(A)$ 
consisting of functions $\phi \colon \W \rightarrow \C$
such that $\Supp(\phi) \subset M_{2m}^* \otimes \Gamma_W$.
We will use the following result proven by Waldspurger.
\par

\begin{prop}[{\cite[Corollary III.2]{W}}]\label{WK}
Let $J_{2m}^V$ be a compact subgroup of $\U(V)$. 
Suppose that 
\begin{itemize}
\item
$J_{2m}^V \supset K_{2m}^V$; 
\item
$\Sc(A)_{M_{2m}}$ is stable by $J_{2m}^V$; 
\item
$(\Sc(A)_{M_{2m}})^{J_m^V} \not= \{0\}$.
\end{itemize}
Then $\Sc(A)^{J_{2m}^V}$ is generated by $(\Sc(A)_{M_{2m}})^{J_{2m}^V}$ as a representation of $\U(W)$.
\end{prop}

We will apply this proposition 
to the compact subgroup $J_{2m}^V$ generated by $K_{2m}^V$ and $E^1 \cap (1+\pp_E^{m})$, 
where the latter is regarded as a subgroup of the center of $\U(V)$.
Namely, $J_{0}^V = K_{0}^V$, and
\[
J_{2m}^V =
\bordermatrix{
&n&1&n \cr
n&\oo_E&\pp_E^m&\oo_E \cr
1&\pp_E^m&1+\pp_E^{m}&\pp_E^m \cr
n&\oo_E&\pp_E^m&\oo_E
} 
\cap \U_{2n+1}
\]
for $2m > 0$.
It is clear that $J_{2m}^V \supset K_{2m}^V$. 
\par

We check the second and third conditions in Proposition \ref{WK}.
\begin{lem}\label{23}
The space $\Sc(A)_{M_{2m}}$ is stable by $J_{2m}^V$ and fixed by $K_{2m}^V$. 
Moreover, $(\Sc(A)_{M_{2m}})^{J_{2m}^V} \not= \{0\}$.
\end{lem}
\begin{proof}
For $t \in \pp_E^{-m}$ and $w \in \Gamma_W$, 
define $\phi_{t,w} \in \Sc(A)$ so that 
$\Supp(\phi_{t,w}) = A + te_0 \otimes w$ and $\phi_{t,w}(te_0 \otimes w) = 1$.
Then $\Sc(A)_{M_{2m}}$ is equal to the $\C$-span of 
\[
\left\{\phi_{t,w} \;\middle|\; t \in \pp_E^{-m},\; w \in \Gamma_W \right\}.
\]
For $k \in J_{2m}^{V}$, write 
$ke_0 = \sum_{i = -n}^n k_i e_i$. 
Then 
\[
k_i \in 
\left\{
\begin{aligned}
&\pp_E^{m} \iif i \not=0, \\
&1+\pp_E^{m} \iif i = 0. 
\end{aligned}
\right.
\]
In particular, we see that $(k-1)te_0 \otimes w \in A$. 
Hence
\begin{align*}
\Supp(\omega_\psi(k)\phi_{t,w}) 
&= k(A + te_0 \otimes w)
\\&= A + (k-1)te_0 \otimes w + te_0 \otimes w
\\&= A + te_0 \otimes w = \Supp(\phi_{t,w}).
\end{align*}
Moreover, 
\begin{align*}
&\omega_\psi(k)\phi_{t,w}(te_0 \otimes w) 
\\&= \omega_\psi(k)\phi_{t,w}\left(kte_0 \otimes w - (k-1)te_0 \otimes w\right) 
\\&= \psi\left( \half{1}\pair{(k-1)te_0 \otimes w, kte_0 \otimes w} \right)
\omega_\psi(k)\phi_{t,w}(kte_0 \otimes w)
\\&= \psi_E\left( \pair{(k-1)te_0,kte_0}_V \cdot \pair{w,w}_W\right)
\phi_{t,w}(te_0 \otimes w)
\\&= \psi_E\left( N_{E/F}(t)(\pair{ke_0,ke_0}_V-\pair{e_0,ke_0}_V) \cdot \pair{w,w}_W\right)
\\&= \psi_E\left( N_{E/F}(t)(1-\overline{k_0}) \pair{w,w}_W \right).
\end{align*}
Hence, 
for $t \in \pp_E^{-m}$, $w \in \Gamma_W$ and $k \in J_{2m}^V$, 
there exists $c \in \C^\times$ such that 
$\omega_\psi(k)\phi_{t,w} = c\phi_{t,w}$.
This shows that $\Sc(A)_{M_{2m}}$ is stable by $J_{2m}^V$. 
Moreover, if $k_0 \in \pp_E^{2m}$ or $\pair{w,w}_W = 0$, then $c=1$. 
Hence we have $(\Sc(A)_{M_{2m}})^{K_{2m}^V} = \Sc(A)_{M_{2m}}$ and $(\Sc(A)_{M_{2m}})^{J_{2m}^V} \not= \{0\}$.
\end{proof}

Therefore, by Proposition \ref{WK}, 
we see that $\Sc(A)^{J_{2m}^V}$ is generated by $(\Sc(A)_{M_{2m}})^{J_{2m}^V}$ as a representation of $\U(W)$.
If $2m > 0$, then $\Sc(A)_{M_{2m}} \supset \Sc(A)_{M_{2m-2}}$.
Let $\Sc(A)_{M_{2m} \setminus M_{2m-2}}$ be the subspace spanned by
\[
\left\{\phi_{t,w} \;\middle|\; \ord(t) = -m,\; w \in \Gamma_W \setminus \varpi\Gamma_W \right\}.
\]
Then we have 
\[
\Sc(A)_{M_{2m}} = \Sc(A)_{M_{2m-2}} \oplus \Sc(A)_{M_{2m} \setminus M_{2m-2}}.
\]

\begin{lem}\label{span}
Suppose that $2m > 0$.
The image $(\Sc(A)_{M_{2m}})^{J_{2m}^V}$ under the projection 
$\Sc(A)_{M_{2m}} \twoheadrightarrow \Sc(A)_{M_{2m} \setminus M_{2m-2}}$
is equal to the one of the subspace spanned by
\[
\left\{ \omega_\psi(k')\phi_{t, f_n} \;\middle|\; \ord(t) = -m,\; k' \in K_0^W \right\}.
\]
Moreover, $\phi_{t,f_n}$ is fixed by $K_{2m}^W$, and $\phi_{t,f_{-n}}$ is fixed by ${}^tK_{2m}^W$.
\end{lem}
\begin{proof}
As we have seen in the proof of Lemma \ref{23}, 
$k \in J_{2m}^V$ acts on $\phi_{t,w}$ by the character 
\begin{align*}
\xymatrix{
J_{2m}^V \ar@{->>}[r] & 1+\pp_E^m \ar@{->}[r] & \C^\times, \\
k \ar@{|->}[r] & k_0 \ar@{|->}[r] & \psi_E\left( N_{E/F}(t)(1-\overline{k_0}) \pair{w,w}_W \right).
}
\end{align*}
Hence the image in question is equal to the one of  the subspace spanned by $\phi_{t,w}$ with 
$\ord(t) = -m$ and $w \in \Gamma_W \setminus \varpi\Gamma_W$ such that $\pair{w,w}_W \in \pp_E^m$.
It means that 
\[
\pair{w,w}_W \equiv \pair{f_n, f_n}_W \bmod \pp_E^{m}.
\]
Note that
\[
K_0^W = \{g \in \U(V) \;|\; g \Gamma_W = \Gamma_W\}
\]
is a hyperspecial maximal compact subgroup of $\U(W)$. 
Hence there exists $k' \in K_0^W$ such that 
$w \equiv k' \cdot f_n \bmod \varpi^{m}\Gamma_W$.
In particular, we have 
\[
te_0 \otimes w - te_0 \otimes k' \cdot f_n \in A.
\]
Hence we can find $c \in \C^\times$ such that $\phi_{t,w} = c\phi_{t, k' \cdot f_n} = c \cdot \omega_\psi(k') \phi_{t,f_n}$.
This shows the first assertion.
\par

Fix $k' \in K_{2m}^W$.
Since $(k'-1)f_n \in \varpi^{m} \Gamma_W$, we have $(k'-1)(te_0 \otimes f_n) \in A$ for $t \in \pp_E^{-m}$. 
Hence $\Supp(\omega_\psi(k') \phi_{t,f_n}) = \Supp(\phi_{t,f_n})$.
Moreover, 
\begin{align*}
&\omega_\psi(k') \phi_{t,f_n}(te_0 \otimes f_n)
\\&= \omega_\psi(k') \phi_{t,f_n}(k'(te_0 \otimes f_n) - (k'-1)(te_0 \otimes f_n))
\\&= 
\psi\left(
\half{1}\pair{ (k'-1)(te_0 \otimes f_n),  k'(te_0 \otimes f_n)}
\right)
\omega_\psi(k') \phi_{t,f_n}(k'(te_0 \otimes f_n))
\\&= 
\psi_E\left(
N_{E/F}(t) \pair{(k'-1)f_n,k'f_n}_W
\right).
\end{align*}
Since $\pair{(k'-1)f_n,k'f_n}_W = \pair{-f_n,k'f_n}_W \in \pp_E^{2m}$, 
we have $\omega_\psi(k') \phi_{t,f_n}(te_0 \otimes f_n) = 1$. 
Therefore, we conclude that $\omega_\psi(k') \phi_{t,f_n} = \phi_{t,f_n}$ for $k' \in K_{2m}^W$.
By a similar calculation, one can prove that 
$\omega_\psi(k') \phi_{t,f_{-n}} = \phi_{t,f_{-n}}$ for $k' \in {}^tK_{2m}^W$.
This completes the proof.
\end{proof}

\subsection{Existence of $K_{2m}^W$-fixed vectors}
Let $\pi$ be an irreducible $\psi_E$-generic tempered representation of $\U(W)$
with the $L$-parameter $\phi_\pi$ and the central character $\omega_\pi$. 
Consider its theta lift $\sigma = \theta_\psi(\pi)$. 
It is an irreducible generic tempered representation of $\U(V)$
with $L$-parameter $\phi_\sigma = \phi_\pi \chi \oplus \1$.
In particular, $c(\phi_\sigma) = c(\phi_\pi)$ so that 
$\sigma^{K_{2m}^V} \not= 0$ for $2m \geq c(\phi_\pi)$ by Theorem \ref{odd}. 
Since $\omega_\sigma = \omega_\pi$, 
we see that $\sigma^{J_{2m}^V} \not= 0$ if $2m \geq c(\phi_\pi)$ and $\omega_\pi|_{1+\pp_E^m} = \1$. 
\par

Set $\omega_\psi = \omega_{\psi,V,W}$.
By the definition of theta lifts, 
we have a $\U(V) \times \U(W)$-equivariant surjective map
\[
\Phi \colon \omega_\psi \twoheadrightarrow \sigma \boxtimes \pi.
\]

\begin{prop}\label{K-fixed}
Set $2m = c(\phi_\pi)$ or $2m = c(\phi_\pi)+1$. 
Suppose that $2m >0$ and that $\omega_\pi$ is trivial on $1+\pp_E^m$.
For any sign $\epsilon \in \{\pm1\}$, there exists $t \in \pp_E^{-m}$ 
such that $\Phi(\phi_{t,f_{\epsilon n}}) \not= 0$. 
In particular, $\pi^{K_{2m}^W} \not= 0$. 
\end{prop}
\begin{proof}
We realize $\omega_\psi$ on the lattice model $\Sc(A)$. 
Since $\Sigma \mapsto \Sigma^{J_{2m}^V}$ is an exact functor on 
the category of smooth representations $\Sigma$ of $\U(V)$, 
we obtain a $\U(W)$-equivariant surjective map
\[
\Phi \colon \Sc(A)^{J_{2m}^V} \twoheadrightarrow \sigma^{J_{2m}^V} \boxtimes \pi.
\]
By Proposition \ref{WK} together with Lemma \ref{23}, 
its restriction to $(\Sc(A)_{M_{2m}})^{J_{2m}^V}$ is still nonzero.
Since $\sigma^{K_{2m-2}^V} = 0$, this map factors through the restriction of the projection 
$\Sc(A)_{M_{2m}} \twoheadrightarrow \Sc(A)_{M_{2m} \setminus M_{2m-2}}$.
Hence by Lemma \ref{span}, 
there exists $t \in E^\times$ with $\ord(t) = -m$ such that $\Phi(\phi_{t,f_n}) \not= 0$. 
Since $\phi_{t,f_n}$ is fixed by $J_{2m}^V \times K_{2m}^W$, 
we have $\Phi(\phi_{t,f_n}) \in \sigma^{J_{2m}^V} \boxtimes \pi^{K_{2m}^W}$ 
so that $\pi^{K_{2m}^W} \not= 0$. 
By the same argument, one can show that $\Phi(\phi_{t,f_{-n}}) \not= 0$ for some $t \in \pp_E^{-m}$. 
\end{proof}

\subsection{Proof of Theorem \ref{even} (3)}\label{sec.proof3}
The goal of the rest of this section is to show that $\pi_\psi^{K_{2m}^W} \not= 0$ 
if $2m = c(\phi_\pi)$ or $2m = c(\phi_\pi)+1$ and if $\omega_\pi$ is trivial on $1+\pp_E^m$. 
If $2m = c(\phi_\pi) = 0$, 
then $\pi$ is unramified (with respect to the hyperspecial maximal compact subgroup $K_0^W$), 
and the Casselman--Shalika formula \cite{CS} shows that $\pi_\psi^{K_{0}^W} \not= 0$.
See Remark \ref{W(w)}.
Hence we may assume that $c(\phi_\pi) > 0$ so that $2m > 0$. 
\par

We need further notations.
Set 
\[
X = \bigoplus_{i=1}^n Ee_i,\quad V_0 = E e_0, \quad X^* = \bigoplus_{i=1}^n Ee_{-i}.
\]
Hence $V = X \oplus V_0 \oplus X^*$. 
For $a \in \GL(X)$, $b \in \Hom(V_0,X)$ and $c \in \Hom(X^*,X)$, 
we define $a^* \in \GL(X^*)$, $b^* \in \Hom(X^*,V_0)$ and $c^* \in \Hom(X^*,X)$ so that
\[
\pair{ax, x'}_V = \pair{x,a^*x'}_V, 
\quad 
\pair{be_0, x'}_V = \pair{e_0,b^*x'}_V, \\
\quad
\pair{cx', x''}_V = \pair{x', cx''}_V
\]
for $x \in X$ and $x',x'' \in X^*$. 
For $a \in \GL(X)$, $b \in \Hom(V_0,X)$ and 
\[
c \in \Herm(X^*,X) = \{c \in \Hom(X^*,X) \;|\; c^*=-c\}, 
\]
we put
\begin{align*}
\bm_X(a) &= \begin{pmatrix}
a && \\
&\1_{V_0}& \\
&& (a^*)^{-1}
\end{pmatrix}, \\
\quad 
\bn_1(b) &= 
\begin{pmatrix}
\1_X & b & -\half{1}bb^*\\
&\1_{V_0}& b^* \\
&& \1_{X^*}
\end{pmatrix}, \\
\bn_2(c) &= 
\begin{pmatrix}
\1_X &  & c \\
&\1_{V_0}& \\
&& \1_{X^*}
\end{pmatrix}.
\end{align*}
These are elements in $\U(V)$.
\par

Similarly, set
\[
Y = \bigoplus_{i=1}^n Ef_i, \quad Y^* = \bigoplus_{i=1}^n Ef_{-i}
\]
so that $W = Y \oplus Y^*$. 
For $a \in \GL(Y)$ and $c \in \Hom(Y^*,Y)$, 
we define $a^* \in \GL(X^*)$ and $c^* \in \Hom(Y^*,Y)$ so that
\[
\pair{ay, y'}_W = \pair{y,a^*y'}_W, 
\quad 
\pair{cy', y''}_W = \pair{y', cy''}_W
\]
for $y \in Y$ and $y',y'' \in Y^*$. 
For $a \in \GL(Y)$ and 
\[
c \in \Herm(Y^*,Y) = \{c \in \Hom(Y^*,Y) \;|\; c^*=-c\}, 
\]
we put
\[
\bm_Y(a) = \begin{pmatrix}
a & \\
& (a^*)^{-1}
\end{pmatrix}, 
\quad 
\bn(c) = 
\begin{pmatrix}
\1_Y & c \\
& \1_{Y^*}
\end{pmatrix}.
\]
These are elements in $\U(W)$. 
\par

Define $a_\delta \in \GL(X)$ by 
\[
a_\delta \colon e_i \mapsto \delta^{-i} e_i
\]
for $-n \leq i \leq n$.
If we fix a nonzero Whittaker functional $l_\sigma \in \Hom_{N_{2n+1}}(\sigma, \psi_E)$ for $\sigma$, 
then $l_{\sigma}' = l_\sigma \circ \sigma(\bm_X(a_\delta))$ is a nonzero Whittaker functional 
with respect to the character $\psi_{E}^\delta \colon N_{2n+1} \rightarrow \C^\times$
given by 
\[
\psi_{E}^\delta(u) = 
\psi_E\left( 
\delta^{-1} \sum_{i=1}^{n} \pair{u e_{i-1}, e_{-i}}_V 
\right). 
\]
This is the generic character considered in \cite{Cheng}. 
\par

Now we fix $t \in E^\times$ with $\ord(t) = -m$ such that 
$\Phi(\phi_{t, f_{-n}}) \not= 0$.
This belongs to $\sigma^{J_{2m}^V} \boxtimes \pi$. 
Note that $\bm_X(t \cdot \1_X) K_{2m}^V \bm_X(t \cdot \1_X)^{-1}$ is the compact group $K_{n,2m}$ 
considered in \cite{Cheng}.
In particular, the Whittaker functional
\[
l'_{\sigma,t} = 
l'_\sigma \circ \sigma(\bm_X(t \cdot \1_X))
= l_\sigma \circ \sigma(\bm_X(a_\delta t))
\]
with respect to 
\[
\psi_{E,t}^\delta \colon N_{2n+1} \ni u \mapsto 
\psi_E^{\delta}(\bm_X(t \cdot \1_X) \cdot u \cdot \bm_X(t \cdot \1_X)^{-1}) \in \C^\times
\]
is nonzero on $\sigma^{K_{2m}^V}$ by \cite[Theorem 1.4, Lemma 7.5]{Cheng}. 
Therefore, the image $\phi_{t, f_{-n}}$ under the composition of $N_{2n+1} \times \U(W)$-equivariant maps 
\[
\xymatrix{
\omega_{\psi} \ar@{->}[r]^-{\Phi} & \sigma \boxtimes \pi \ar@{->}[rr]^-{l'_{\sigma,t} \otimes \id} 
&& \psi_{E,t}^\delta \boxtimes \pi 
}
\]
is nonzero.
\par

By the same argument as the proof of \cite[Proposition 2.3]{MR}, 
one can prove that
the maximal quotient of $\omega_\psi$ on which $N_{2n+1}$ acts by $\psi_{E,t}^\delta$
is isomorphic to the compact induction $\ind_{N_{2n}'}^{\U(W)}(\mu)$, 
where $N_{2n}'$ is the unipotent radical of the Borel subgroup of $\U(W)$
stabilizing the flag 
\[
E f_1 \subset E f_1 \oplus E f_2 \subset \dots \subset E f_1 \oplus \dots \oplus E f_n = Y, 
\]
and $\mu$ is a character of $N_{2n}'$ given by
\[
\mu(u) = \psi_E \left( \sum_{i=1}^n \pair{u f_{i+1}, f_{-i}} + N_{E/F}(t)\pair{u f_{-n}, f_{-n}} \right).
\]
Here, we note that $N_{2n}'$ differs from $N_{2n}$ defined in Section \ref{sec.rep}. 
\par

Hence the map 
\[
\xymatrix{
\omega_{\psi} \ar@{->}[r]^-{\Phi} & \sigma \boxtimes \pi \ar@{->}[rr]^-{l_{\sigma,t}' \otimes \id} 
&& \psi_{E,t}^\delta \boxtimes \pi 
}
\]
factors through $\omega_\psi \rightarrow \ind_{N_{2n}'}^{\U(W)}(\mu)$.
Namely, we have a nonzero $\U(W)$-equivariant map
\[
\ind_{N_{2n}'}^{\U(W)}(\mu) \rightarrow \pi. 
\]
The following is a key lemma, which will be proven in Section \ref{sec.key} below. 

\begin{lem}\label{supp}
Let $\tl{F}_{t,f_{-n}} \in \ind_{N_{2n}'}^{\U(W)}(\mu)$ be the image of $\phi_{t, f_{-n}} \in \Sc(A)$.
Then $\tl{F}_{t,f_{-n}}$ is right ${}^tK_{2m}^W$-invariant and 
\[
\Supp(\tl{F}_{t,f_{-n}}) = N_{2n}' \cdot {}^tK_{2m}^W. 
\]
\end{lem}

Note that having a $\U(W)$-equivariant map 
\[
\ind_{N_{2n}'}^{\U(W)}(\mu) \rightarrow \pi
\]
is equivalent to giving a $\U(W)$-equivariant map 
\[
\pi^\vee \rightarrow \Ind_{N_{2n}'}^{\U(W)}(\mu^{-1}). 
\]
These are related as follows.
Suppose that $\ind_{N_{2n}'}^{\U(W)}(\mu) \ni \tl{F} \mapsto v \in \pi$ 
corresponds to $\pi^\vee \ni v' \mapsto W \in \Ind_{N_{2n}'}^{\U(W)}(\mu^{-1})$.
Then 
\[
(v,v')_\pi = \int_{N_{2n}' \bs \U(W)} \tl{F}(g) W(g) dg.
\]
\par

By Lemma \ref{supp}, 
there exists $\tl{F} \in \ind_{N_{2n}'}^{\U(W)}(\mu)$ such that 
\begin{itemize}
\item
its image $v$ in $\pi$ is nonzero; 
\item
$\tl{F}$ is right ${}^tK_{2m}^W$-invariant; 
\item
$\Supp(\tl{F}) = N_{2n}' \cdot {}^tK_{2m}^W$.
\end{itemize}
Hence $v \in \pi^{{}^tK_{2m}^W}$. 
One can take $v' \in (\pi^\vee)^{{}^tK_{2m}^W}$ such that $(v,v')_\pi \not= 0$.
Let $W \in \Ind_{N_{2n}'}^{\U(W)}(\mu^{-1})$ be the image of $v'$. 
Then $W$ is right ${}^tK_{2m}^W$-invariant, and 
\[
0 \not= (v,v')_\pi = \int_{N_{2n}' \bs \U(W)} \tl{F}(g) W(g) dg
= c\tl{F}(\1) W(\1)
\]
for some constant $c > 0$.
Hence $W(\1) \not= 0$.
Moreover, since $\bv(0,0;z) \in N_{2n}'$, we have
\[
W(\bv(0,0;z)) = \mu^{-1}(\bv(0,0;z)) W(\1) = \psi^{-1}(N_{E/F}(t)z) W(\1)
\]
for $z \in F$.
Therefore, via $v' \mapsto W \mapsto W(\1)$, we conclude that 
\[
(\pi^\vee)^{{}^tK_{2m}^W}_{\psi'^{-1}} \not= 0, 
\]
where we put $\psi'(z) = \psi(N_{E/F}(t)z)$.
Since 
\[
{}^tK_{2m}^W = 
\begin{pmatrix}
t && \\ 
&\1_{2n-2}& \\
&&\overline{t}^{-1}
\end{pmatrix}^{-1}
K_{2m}^W
\begin{pmatrix}
t && \\ 
&\1_{2n-2}& \\
&& \overline{t}^{-1}
\end{pmatrix},
\]
as in Section \ref{sec.jacobi}, 
we have 
\[
(\pi^\vee)^{K_{2m}^W}_{\psi^{-1}} \cong (\pi^\vee)^{{}^tK_{2m}^W}_{\psi'^{-1}} \not= 0.
\]
\par

Since $\pi$ is $\psi_E$-generic if and only if $\pi^\vee$ is $\psi_E^{-1}$-generic, 
by replacing $\pi$ and $\psi$ with $\pi^\vee$ and $\psi^{-1}$, respectively,
we conclude that 
\[
\pi^{K_{2m}^W}_\psi \not= 0.
\]
Therefore, Theorem \ref{even} (3) is reduced to proving Lemma \ref{supp}.

\subsection{Mixed model}
To show Lemma \ref{supp}, 
we review the argument in the proof of \cite[Proposition 2.3]{MR}. 
For this, we use another model of the Weil representation $\omega_\psi = \omega_{\psi,V,W}$ of $\U(V) \times \U(W)$. 
It is known that the Weil representation $\omega_\psi$ can be realized on the space 
$\Sc(X^* \otimes W) \otimes \Sc(V_0 \otimes Y^*)$, 
which is called a \emph{mixed model}. 
See e.g., \cite[Section 7.4]{GI}.
Let us recall some formulas for the action of $\U(V) \times \U(W)$ on this space. 
\par

For $\varphi_1 \otimes \varphi_2 \in \Sc(X^* \otimes W) \otimes \Sc(V_0 \otimes Y^*)$
and $(x,y) \in (X^* \otimes W) \times (V_0 \otimes Y^*)$, 
\begin{align*}
&\omega_\psi(g)(\varphi_1 \otimes \varphi_2)(x,y) 
= \varphi_1(g^{-1}x) \cdot \omega_\psi^0(g)\varphi_2(y), 
\quad g \in \U(W), \\
&\omega_\psi(h_0)(\varphi_1 \otimes \varphi_2)(x,y)
= \varphi_1(x) \cdot \omega_\psi^0(h_0)\varphi_2(y), 
\quad h_0 \in \U(V_0), \\
&\omega_\psi(\bm_X(a))(\varphi_1 \otimes \varphi_2)(x,y)
= \chi_W(\det a) |\det a|^n \varphi_1(a^* x) \cdot \varphi_2(y),
\quad a \in \GL(X), \\
&\omega_\psi(\bn_1(b))(\varphi_1 \otimes \varphi_2)(x,y)
= \varphi_1(x) \cdot \rho_\psi^0(b^*x, 0)\varphi_2(y), 
\quad b \in \Hom(V_0,X), \\
&\omega_\psi(\bn_2(c))(\varphi_1 \otimes \varphi_2)(x,y)
= \psi\left(\half{1}\pair{cx,x}\right) \varphi_1(x) \cdot \varphi_2(y), 
\quad c \in \Herm(X^*,X).
\end{align*}
Here, $\Sc(V_0 \otimes Y^*)$ is regarded as the Schr{\"o}dinger model of 
\begin{itemize}
\item
the irreducible representation $\rho_\psi^0$ of the Heisenberg group $H(V_0 \otimes W)$ on $\Sc(V_0 \otimes Y^*)$ 
with the central character $\psi$; 
and 
\item
the Weil representation $\omega_\psi^0$ of $\U(V_0) \times \U(W)$.
\end{itemize}
Hence, for $\varphi_2 \in \Sc(V_0 \otimes Y^*)$ and $y \in V_0 \otimes Y^*$, 
we have
\[
\rho_\psi^0((y_+ + y_-,t))\varphi_2(y) 
= \psi\left( t + \pair{y,y_+} + \half{1}\pair{y_-,y_+} \right) \varphi_2(y+y_-)
\]
for $y_+ \in V_0 \otimes Y$ and $y_- \in V_0\otimes Y^*$, 
and
\begin{align*}
&\omega_\psi^0(\bm_Y(a))\varphi_2(y) 
= \chi(\det a)|\det a|^\half{1}\varphi(a^*y), 
\quad a \in \GL(Y), \\
&\omega_\psi^0(\bn(c))\varphi_2(y)
= \psi\left(\half{1} \pair{cy,y}\right) \varphi_2(y)
\quad c \in \Herm(Y^*,Y). 
\end{align*}
Moreover, $\omega_\psi^0(J_{2n})\varphi_2$ is given by a Fourier transform of $\varphi_2$.
For more precision, see \cite[Section 7.4]{GI}.
\par

For $\varphi_1 \otimes \varphi_2 \in \Sc(X^* \otimes W) \otimes \Sc(V_0 \otimes Y^*)$, 
define 
\[
F_{\varphi_1 \otimes \varphi_2}(g) 
= \varphi_1(g^{-1} x_0) \cdot \omega_\psi^0(g)\varphi_2(y_0), 
\]
where we set 
\[
x_0 = \sum_{i=1}^n \frac{1}{2\delta} e_{-i} \otimes f_{n+1-i}, 
\quad
y_0 = t e_0 \otimes f_{-n}.
\]
Let $Q_{2n} = M_{2n,S}N_{2n,S}$ 
be the Siegel parabolic subgroup of $\U(W)$ stabilizing $Y$, 
where $M_{2n,S} = \{\bm_Y(a) \;|\; a \in \GL(Y)\}$ is its Levi subgroup, 
and $N_{2n,S}$ is its unipotent radical.
Note that $N_{2n,S} \subset N_{2n}'$.
We regard $\mu$ as a character of $N_{2n,S}$ by the restriction.
For $u \in N_{2n,S}$, since $u^{-1}x_0 = x_0$ and 
\begin{align*}
\psi\left(\half{1}\pair{uy_0,y_0}\right) 
&=\psi_E\left(\pair{te_0, te_0}_V \pair{uf_{-n},f_{-n}}_W \right) 
\\&= \psi_E\left(N_{E/F}(t) \pair{uf_{-n},f_{-n}}_W \right) = \mu(u), 
\end{align*}
we see that $F_{\varphi_1 \otimes \varphi_2}(g) \in \ind_{N_{2n,S}}^{\U(W)}(\mu)$.
Note that $\bn_2(c)$ acts trivially on $F_{\varphi_1 \otimes \varphi_2}$ for $c \in \Herm(X^*,X)$ 
since $Y$ is totally isotropic. 
On the other hand, for $b \in \Hom(V_0,X)$, 
since $\bn_1(b)$ commutes with $g \in \U(W)$, 
we see that 
\begin{align*}
F_{\omega_\psi(\bn_1(b))(\varphi_1 \otimes \varphi_2)}(g)
&= 
\omega_\psi(g) \circ \omega_\psi(\bn_1(b)) (\varphi_1\otimes \varphi_2) (x_0,y_0) 
\\&=
\omega_\psi(\bn_1(b)) \circ \omega_\psi(g) (\varphi_1\otimes \varphi_2) (x_0,y_0) 
\\&=
\rho_\psi^0(b^*x_0,0) \circ \omega_\psi(g) (\varphi_1\otimes \varphi_2) (x_0,y_0)
\\&=
\psi(\pair{y_0,b^*x_0}) \omega_\psi(g) (\varphi_1\otimes \varphi_2) (x_0,y_0). 
\end{align*}
Since $\overline{\delta} = -\delta$ and $\pair{f_{-n},f_n}_W = -1$, we have
\begin{align*}
\psi(\pair{y_0,b^*x_0}) 
&= \psi_E\left(
\sum_{i=1}^n \pair{te_0,b^* \delta^{-1} e_{-i}}_V \pair{f_{-n}, f_{n+i-1}}_W 
\right)
\\&= \psi_E( \delta^{-1} t \pair{be_0,e_{-1}}_V) = \psi_{E,t}^\delta(\bn_1(b)). 
\end{align*}
Hence $\bn_1(b)$ acts on $F_{\varphi_1 \otimes \varphi_2}$ by $\psi_{E,t}^\delta$.
\par

Define a map 
\[
\ind_{N_{2n,S}}^{\U(W)}(\mu) \rightarrow \ind_{N_{2n}'}^{\U(W)}(\mu)
\]
by 
\[
F \mapsto \tl{F}(g) = \int_{N_{2n,S} \bs N_{2n}'}F(ug) \mu(u)^{-1} du.
\]
Then by the same argument as in \cite[Proposition 2.3]{MR}, 
one can prove that the map 
$\varphi_1 \otimes \varphi_2 \mapsto \tl{F}_{\varphi_1 \otimes \varphi_2}$
realizes an isomorphism between 
the maximal quotient of $\omega_\psi$ on which $N_{2n+1}$ acts by $\psi_{E,t}^\delta$
and $\ind_{N_{2n}'}^{\U(W)}(\mu)$.

\subsection{Proof of Lemma \ref{supp}}\label{sec.key}
In this subsection, we prove Lemma \ref{supp}.
To do this, we relate two models of the Weil representation.
\par

Let $\varphi_1^0 \in \Sc(X^* \otimes W)$ and $\varphi_2^0 \in \Sc(V_0 \otimes Y^*)$ be the characteristic functions of 
\[
\left(\bigoplus_{i=1}^n \oo_Ee_{-i} \right) \otimes \left( \bigoplus_{i=1}^n \oo_Ef_{i} \oplus \bigoplus_{i=1}^n \oo_Ef_{-i} \right), 
\quad \oo_E e_0 \otimes \left( \bigoplus_{i=1}^n \oo_Ef_{-i} \right),
\]
respectively. 
Then the action $\rho = \rho_\psi$ of the Heisenberg group $H(\W)$ on $\varphi_1^0 \otimes \varphi_2^0$ satisfies that 
\[
\rho(a,t)(\varphi_1^0 \otimes \varphi_2^0) = \psi(t) \cdot \varphi_1^0 \otimes \varphi_2^0
\]
for $(a,t) \in A \oplus F$.
Moreover, the lattice model $\Sc(A)$ and the mixed model $\Sc(X^* \otimes W) \otimes \Sc(V_0 \otimes Y^*)$ are related by the isomorphism
\begin{align*}
\Sc(A) &\xrightarrow{\sim} \Sc(X^* \otimes W) \otimes \Sc(V_0 \otimes Y^*),\\
\phi &\mapsto \int_{(A \oplus F) \bs H(\W)} \phi(h) \rho(h)^{-1}(\varphi_1^0 \otimes \varphi_2^0)(x,y) dh.
\end{align*}
In particular, $\phi_{t,f_{-n}} \in \Sc(A)$ corresponds to 
\[
\rho(t e_0 \otimes f_{-n},0)^{-1} (\varphi_1^0 \otimes \varphi_2^0)(x,y)
= \varphi_1^0(x) \cdot \rho_\psi^0(t e_0 \otimes f_{-n},0)^{-1} \varphi_2^0(y)
\]
in $\Sc(X^* \otimes W) \otimes \Sc(V_0 \otimes Y^*)$
since $\Supp(\phi_{t,f_{-n}}) = (A + te_0 \otimes f_{-n}) \oplus F$.
Therefore, under the map 
\[
\Sc(A) \rightarrow \ind_{N_{2n}'}^{\U(W)}(\mu)
\]
obtained above, 
the image of $\phi_{t,f_{-n}}$ is 
$\tl{F}_{\rho(t e_0 \otimes f_{-n},0)^{-1} (\varphi_1^0 \otimes \varphi_2^0)}$. 
\par

Now we prove Lemma \ref{supp}. 
\begin{proof}[Proof of Lemma \ref{supp}]
First, we consider $F_{\rho(t e_0 \otimes f_{-n},0)^{-1} (\varphi_1^0 \otimes \varphi_2^0)}$. 
Note that it is left $N_{2n,S}$-invariant and right ${}^tK_{2m}^W$-invariant.
We claim that 
if 
\[
F_{\rho(t e_0 \otimes f_{-n},0)^{-1} (\varphi_1^0 \otimes \varphi_2^0)}(g) \not= 0,
\] 
then 
\[
g \in N_{2n,S} \cdot \bm_Y(a) \cdot {}^tK_{2m}^W
\]
for some $a \in \GL(Y) \cong \GL_n(E)$ such that 
$a^{-1} \in \M_n(\oo_E)$ and 
\[
a^*f_{-n}-f_{-n} \in \bigoplus_{i=1}^n \pp_E^m f_{-i}.
\]
\par

By the Iwasawa decomposition, we have 
$\U(W) = Q_{2n} K_0^W$.
Let $K_S$ and $K_M$ be subgroups of $K_0^W$ defined by 
\[
K_S = 
\bordermatrix{
&n&n \cr
n&\oo_E&\oo_E\cr
n&\pp_E&\oo_E
} \cap \U_{2n}, \quad
K_M = 
\bordermatrix{
&1&2n-2&1 \cr
1&\oo_E&\oo_E&\oo_E\cr
2n-2&\pp_E&\oo_E&\oo_E\cr
1&\pp_E&\pp_E&\oo_E
} \cap \U_{2n}.
\]
By the Bruhat decomposition for a finite unitary group over $\oo_F/\pp_F$, we have 
\begin{align*}
K_0^W &= K_S K_M \cup K_S J_{2n}^{-1} K_M
\\&= (K_S \cap Q_{2n}) K_M \cup (K_S \cap Q_{2n}) J_{2n}^{-1} K_M.
\end{align*}
Since $J_{2n} \in K_0^W$ and $J_{2n}^{-1} K_M J_{2n} = {}^t K_M$, 
by the multiplication of $J_{2n}$ from the right, 
we have 
\[
K_0^W = (K_S \cap Q_{2n}) J_{2n} {}^tK_M \cup (K_S \cap Q_{2n}) {}^tK_M.
\]
Hence 
\begin{align*}
\U(W) &= Q_{2n} J_{2n} {}^tK_M \cup Q_{2n} {}^tK_M
\\&= N_{2n,S}M_{2n,S} J_{2n} {}^tK_M \cup N_{2n,S}M_{2n,S} {}^tK_M.
\end{align*}
Therefore, we may assume that 
$g = \bm_Y(a)J_{2n}k$ or $g = \bm_Y(a)k$
for some $a \in \GL(Y)$ and $k \in {}^tK_M$.
\par

Assume that $g = \bm_Y(a)J_{2n}k$ is in the former case.
Since $\varphi_1^0$ and $\varphi_2^0$ are fixed by $K_{0}^W$, 
and since $\omega_\psi^0(g) \circ \rho_\psi^0(h) \circ \omega_\psi^0(g)^{-1} = \rho_\psi^0(gh)$, 
we have
\begin{align*}
&F_{\rho(t e_0 \otimes f_{-n},0)^{-1} (\varphi_1^0 \otimes \varphi_2^0)}(g)
\\&= \omega_\psi(g) \circ 
\rho(t e_0 \otimes f_{-n},0)^{-1} (\varphi_1^0 \otimes \varphi_2^0)(x_0,y_0)
\\&= \varphi_1(g^{-1}x_0) \cdot 
\omega_\psi^0(g) \rho_\psi^0(t e_0 \otimes f_{-n},0)^{-1}\varphi_2^0(y_0)
\\&= \varphi_1(a^{-1}x_0) \cdot 
\omega_\psi^0(\bm_Y(a)) \rho_\psi^0(t e_0 \otimes J_{2n} k f_{-n},0)^{-1} \varphi_2^0(y_0). 
\end{align*}
Note that $\varphi_1(a^{-1}x_0) \not= 0$ if and only if $a^{-1} \in \M_n(\oo_E)$. 
On the other hand, since $k \in {}^tK_M$, 
if we write $J_{2n} k f_{-n} = y+y^*$ with $y \in Y$ and $y^* \in Y^*$, 
then $y^* \in \oplus_{i=1}^n \pp_E f_{-i}$. 
Up to a nonzero constant, 
$\omega_\psi^0(\bm_Y(a)) \rho_\psi^0(t e_0 \otimes J_{2n} k f_{-n},0)^{-1} \varphi_2^0(y_0)$
is equal to 
\[
\varphi_2^0( te_0 \otimes (a^*f_{-n} - y^*) ).
\]
If this is nonzero, then we must have 
$t(a^*f_{-n} - y^*) \in \oplus_{i=1}^n \oo_E f_{-i}$.
When $a^{-1} \in \M_n(\oo_E)$, this implies that
\[
f_{-n} \in (a^*)^{-1}y^* + \oplus_{i=1}^n \pp_E^m f_{-i} \subset \oplus_{i=1}^n \pp_E f_{-i}.
\]
This is impossible.
Hence we have $F_{\rho(t e_0 \otimes f_{-n},0)^{-1} (\varphi_1^0 \otimes \varphi_2^0)}(g) = 0$.
\par

Next, we assume that $g = \bm_Y(a)k$ is in the latter case.
By the Iwahori decomposition, we may further assume that $kf_{-n} = f_{-n}$.
Then 
\begin{align*}
&F_{\rho(t e_0 \otimes f_{-n},0)^{-1} (\varphi_1^0 \otimes \varphi_2^0)}(g)
\\&= \omega_\psi(g) \circ \rho(t e_0 \otimes f_{-n},0)^{-1} (\varphi_1^0 \otimes \varphi_2^0)(x_0,y_0)
\\&= \varphi_1(g^{-1}x_0) \cdot 
\omega_\psi^0(g) \rho_\psi^0(t e_0 \otimes f_{-n})^{-1}\varphi_2^0(y_0)
\\&= \varphi_1(a^{-1}x_0) \cdot 
\omega_\psi^0(\bm_Y(a)) \rho_\psi^0(t e_0 \otimes kf_{-n})^{-1} \varphi_2^0(y_0)
\\&= \varphi_1(a^{-1}x_0) \cdot 
\omega_\psi^0(\bm_Y(a)) \rho_\psi^0(t e_0 \otimes f_{-n})^{-1} \varphi_2^0(y_0). 
\end{align*}
Up to a nonzero constant, it is equal to 
\[
\varphi_1(a^{-1}x_0) \cdot \varphi_2^0( te_0 \otimes (a^*f_{-n}-f_{-n}) ).
\]
If this is nonzero, then $a^{-1} \in \M_n(\oo_E)$ and 
\[
a^*f_{-n}-f_{-n} \in \bigoplus_{i=1}^n \pp_E^m f_{-i}. 
\]
This proves the claim. 
\par

Now we consider $\tl{F}_{\rho(t e_0 \otimes f_{-n},0)^{-1} (\varphi_1^0 \otimes \varphi_2^0)}$.
Note that it is left $N_{2n}'$-invariant and right ${}^tK_{2m}^W$-invariant.
Suppose that $\tl{F}_{\rho(t e_0 \otimes f_{-n},0)^{-1} (\varphi_1^0 \otimes \varphi_2^0)}(g) \not= 0$. 
By the claim, 
we may assume that $g = \bm_Y(a)$ with $a \in \GL(Y)$ satisfying the conditions in the claim.
By the Iwasawa decomposition, 
we may further assume that $a = a_d a_0$ such that 
\begin{itemize}
\item
$\pair{a_d f_i, f_{-j}}_W = \varpi^{\lambda_i} \delta_{i,j}$ for some $\lambda_i \in \Z$; 
\item
$a_0 \in \GL_n(\oo_E)$ via $\GL(Y) \cong \GL_n(E)$.
\end{itemize}
Since $a^{-1} \in \M_n(\oo_E)$, we have $\lambda_i \leq 0$ for $1 \leq i \leq n$. 
Note that 
\[
a^* f_{-n}-f_{-n} = a_0^* a_d^* f_{-n}-f_{-n} = a_0^* \varpi^{\lambda_n}f_{-n}-f_{-n}. 
\]
Since this is in $\oplus_{i=1}^n \pp_E^m f_{-i}$, 
we have $\lambda_n = 0$ and $\bm_Y(a_0) \in {}^tK_{2m}^W$.
Hence we may assume that $a_0 = \1_X$, i.e., $g = \bm_Y(a_d)$.
For $2 \leq i \leq n$ and $x \in \oo_E$, define $u_i \in N_{2n}'$ so that 
\[
u_i f_j - f_j = \left\{
\begin{aligned}
&x \cdot f_{i-1} \iif j=i, \\
&0 \iif j\not=i.
\end{aligned}
\right. 
\]
Then $u_i \in {}^tK_{2m}^W$.
Hence 
\begin{align*}
0 &\not= \tl{F}_{\rho(t e_0 \otimes f_{-n},0)^{-1} (\varphi_1^0 \otimes \varphi_2^0)}(\bm_Y(a_d)) 
\\&=  \tl{F}_{\rho(t e_0 \otimes f_{-n},0)^{-1} (\varphi_1^0 \otimes \varphi_2^0)}(\bm_Y(a_d)u_i) 
\\&= \mu(\bm_Y(a_d) u_i \bm_Y(a_d)^{-1}) \tl{F}_{\rho(t e_0 \otimes f_{-n},0)^{-1} (\varphi_1^0 \otimes \varphi_2^0)}(\bm_Y(a_d)) 
\end{align*}
so that $\mu(\bm_Y(a_d) u_i \bm_Y(a_d)^{-1}) = 1$ for any $x \in \oo_E$.
Note that 
\begin{align*}
\mu(\bm_Y(a_d) u_i \bm_Y(a_d)^{-1}) 
&= \psi_E( \pair{\bm_Y(a_d) u_i \bm_Y(a_d)^{-1} f_{i}, f_{-i+1}} )
\\&= \psi_E( \varpi^{\lambda_{i-1}-\lambda_i}x).
\end{align*}
Hence $\psi_E( \varpi^{\lambda_{i-1}-\lambda_i}x) = 1$ for any $x \in \oo_E$. 
This implies that $\lambda_{i-1} \geq \lambda_i$.
In conclusion, we have 
\[
0 \geq \lambda_1 \geq \dots \geq \lambda_{n-1} \geq \lambda_n = 0
\]
so that $\lambda_1=\dots=\lambda_n=0$.
This means that $a_d = \1_X$.
This completes the proof of Lemma \ref{supp}.
\end{proof}


\end{document}